\documentclass[8pt,a4paper]{article}

\usepackage[utf8]{inputenc}
\usepackage[english]{babel}

\usepackage{amsmath}
\usepackage{amssymb}
\usepackage{mathtools}
\usepackage{amsthm}

\usepackage[shortlabels]{enumitem}
\usepackage[title]{appendix}
\usepackage{optidef}

\usepackage{algorithm}
\usepackage{algpseudocode}

\usepackage{bbm}
\usepackage[numbers,sort&compress]{natbib}

\usepackage[pdftex]{graphicx}
\usepackage{scalerel}
\usepackage{comment}
\usepackage{setspace}
\usepackage{indentfirst}
\usepackage{makeidx}
\usepackage[nottoc]{tocbibind}
\usepackage{courier}
\usepackage{listings}
\usepackage{titletoc}

\usepackage[font=small,format=plain,labelfont=bf,textfont=it]{caption}
\usepackage[usenames,svgnames,dvipsnames]{xcolor}

\usepackage[a4paper,top=2.54cm,bottom=2.0cm,left=2.0cm,right=2.54cm]{geometry}

\usepackage{array}
\usepackage{tikz}
\usetikzlibrary{shapes,shapes.geometric,arrows,positioning}

\usepackage{emptypage}
\usepackage{doi}
\usepackage[all]{hypcap}

\usepackage{emptypage}  
\fontsize{60}{62}\usefont{OT1}{cmr}{m}{n}{\selectfont}
\definecolor{Floresta}{rgb}{0.13,0.54,0.13}
\definecolor{Rosadif}{rgb}{0.6196,0.1254,0.6980} 
\definecolor{Rosa}{rgb}{0.8705, 0.0862, 0.5058}
\definecolor{Verdeest}{rgb}{0.0549, 0.7529, 0.1450} 
\definecolor{Verde}{rgb}{0.1254,0.70196,0.1254} 
\definecolor{OrangeRed}{rgb}{0.9294,0.0509,0.3529} 
\definecolor{Laranja}{rgb}{1,0.4352941176, 0}  
\definecolor{Amarelo}{rgb}{0.9215, 0.8196, 0.18}

\newtheorem{theorem}{Theorem}[section]
\newtheorem{example}{Example}[section]

\newtheorem{lemma}{Lemma}[section]
\newtheorem{proposition}{Proposition}[section]
\newtheorem{assumption}{Assumption}[section]
\newtheorem{definition}{Definition}[section]
\newtheorem{remark}{Remark}

\usepackage{tikz}
\usetikzlibrary{shapes, shapes.geometric, arrows, positioning}
\tikzstyle{nodo} = [rectangle, rounded corners, minimum width=1cm, minimum height=0.5cm,text centered, draw=black, fill=green!30]
\tikzstyle{arrow} = [thick,->,>=stealth]

\everymath{\displaystyle}

\title{Augmented Lagrangian methods for nonlinear semidefinite programming with complementarity constraints}
\author{Daiana O. Santos\thanks{Department of Actuarial Science, Paulista School of Politics, Economics and Business, Federal University of São Paulo, Osasco-SP, Brazil.}
 \and Carina M. Costa\thanks{Department of Mathematics, State University of Maringá, Maringá-PR, Brazil.} 
  \and Evelin H. M. Krulikovski\thanks{Department of Mathematics, Federal University of Paraná, Curitiba-PR, Brazil.} 
  \and Marina Geremia\thanks{Câmpus Chapecó, Instituto Federal de Santa Catarina (IFSC), Chapecó, Brazil.}    
}

\date{\today}
\begin{document}

\maketitle

\begin{abstract}We consider nonlinear semidefinite programming problems with complementarity constraints (SDCMPCC), a class of highly degenerate problems where classical optimality conditions typically fail. In  this context, weak stationarity conditions have been developed to address their degeneracy. While these notions are well understood, their algorithmic implications remain largely unexplored in semidefinite complementarity settings. We introduce a reformulation based on the spectral decomposition of the complementarity structure, which preserves local solutions and enables a tractable analysis.
Within this framework, we analyze an augmented Lagrangian method for SDCMPCC and prove that under a suitable extension of Robinson’s condition, every accumulation point of the generated sequence is W-stationary, or C-stationary under a stricter condition for solving the subproblems.
\end{abstract}

\noindent {\bf Keywords:} 
Nonlinear Semidefinite Programming,
Complementarity Constraints,
Optimality Conditions
{\small
	\\[2mm]
	}

\section{Introduction}
\begin{sloppypar}
Mathematical programs with complementarity constraints (MPCC) arise naturally in engineering design, economics, and equilibrium modeling. These problems are notoriously difficult because standard \emph{constraint qualifications} (CQ) are systematically violated, which invalidates classical Karush--Kuhn--Tucker (KKT) theory and present significant challenges to the development of reliable algorithms; see, e.g.,~\cite{dempe2003annotated,ferris1997engineering, flegel2005guignard,  IzmailovSolodovUskov2012, luo1996}. A central advance in overcoming these difficulties has been the replacement of traditional KKT conditions by stationarity hierarchies, including weak, Clarke, Mordukhovich, and strong stationarity, as well as by sequential notions of optimality, such as the Approximate KKT (AKKT) condition~\cite{akkt}. These latter provide necessary conditions compatible with the stopping criteria of modern algorithms and remain valid even when constraint qualifications fail.
\end{sloppypar}

Sequential optimality conditions have attracted considerable attention in nonlinear optimization because of their deep connection to the behavior of iterative methods. In contrast to classical KKT-type conditions, which often require constraint qualifications, sequential conditions such as AKKT are genuine necessary conditions satisfied by every local minimizer of a smooth problem. The fact that augmented Lagrangian methods, interior-point techniques, and sequential quadratic programming converge to AKKT points has enabled the analysis of convergence without assuming strong regularity notions. This, in turn, has motivated the extension of sequential optimality concepts to a wide range of settings, including multiobjective optimization~\cite{p1,p2}, nonlinear semidefinite programming~\cite{daiana}, nonlinear complementarity problems~\cite{leo, alberto}, generalized Nash equilibrium problems~\cite{frank}, optimization of discontinuous functions~\cite{birgin}, Banach space optimization~\cite{kanzow}, and variational inequalities~\cite{frank2,variacional,kanzow2}.

Despite these advances, the integration of sequential optimality conditions with semidefinite complementarity structures has not yet been systematically studied. Semidefinite cone mathematical programming with complementarity constraints problems, hereafter referred to as SDCMPCC, extend classical MPCC to settings where variables are symmetric matrices and feasibility is governed by the semidefinite order. This framework leads to additional degeneracy phenomena and increased analytical challenges, since not only do standard constraint qualifications fail, but the complementarity structure itself is embedded in a nonlinear spectral geometry. Earlier works, such as those by Yan and Fukushima~\cite{yan2011smoothing} and Ding, Sun, and Ye~\cite{dsun}, established the first theoretical foundations for stationarity concepts in SDCMPCC and motivated subsequent investigations into their structure and applications, including rank-constrained correlation matrix problems, bilinear matrix inequalities, and equilibrium models with uncertain data. More recently, Madariaga and Ramírez~\cite{hector} introduced a facial characterization of stationarity, emphasizing the structural complexity of semidefinite complementarity.

\textbf{Main contributions of this work:}
\begin{enumerate}
    \item We introduce a tightened reformulation of nonlinear semidefinite programs with complementarity constraints based on the spectral decomposition of the complementarity structure. This reformulation separates active, inactive, and biactive spectral components, preserves local solutions, and enables a tractable first-order optimality analysis.
   \item We introduce the notions of W- and C-stationarity tailored to SDCMPCC, extending classical MPCC concepts to the semidefinite setting.
\item We develop their sequential counterparts, AW- and AC-stationarity, and prove that they constitute genuine necessary optimality conditions aligned with the behavior of modern algorithms.
    \item We analyze an augmented Lagrangian method for SDCMPCC and prove that every accumulation point of the generated sequence is AW-stationary. Moreover, under a mild additional assumption, it is also AC-stationary, and, under the SDCMPCC-Robinson constraint qualification, it satisfies C-stationarity.
\end{enumerate}
\subsection{Structure of the Paper}
The paper is organized to progressively develop both the theoretical and algorithmic aspects of nonlinear semidefinite programming with complementarity constraints.
Subsection~\ref{notation} introduces the notation that will be used throughout the paper.
In Section~\ref{sessao2}, we review first-order optimality conditions for nonlinear semidefinite programming (NLSDP), with particular emphasis on sequential optimality concepts.
Section~\ref{sessao3} presents the formulation of nonlinear semidefinite programs with complementarity constraints (SDCMPCC) and discusses the intrinsic degeneracy of these problems, highlighting the systematic failure of classical constraint qualifications and motivating the need for alternative stationarity concepts.
In Section~\ref{sessao4}, we introduce a tightened reformulation based on the spectral decomposition of the complementarity structure. 
Section~\ref{sessao5} is devoted to the development of stationarity notions for SDCMPCC. 
We introduce the concepts of W- and C-stationarity, together with their 
sequential counterparts, namely AW- and AC-stationarity.
In Section~\ref{algorithm}, we investigate the algorithmic consequences of the proposed framework. 
Finally, Section~\ref{comentarios} concludes the paper.
\subsection{Notation}\label{notation}
Throughout the paper, all matrices are real.  Let $f: \mathbb{R}^n \to \mathbb{R}$ be $C^1$, and $\nabla f(x)$ denotes its gradient at $x$. We denote by $\mathbb{S}^m$ the space of $m \times m$ real symmetric matrices and by $\mathbb{O}^m$ the set of orthogonal matrices. 
For $X \in \mathbb{S}^m$, $X \succeq 0$ (resp.\ $X \preceq 0$) denotes that $X$ is positive semidefinite (resp.\ negative semidefinite), while $X \succ 0$ (resp.\ $X \prec 0$) denotes that $X$ is positive definite (resp.\ negative definite). Accordingly, $\mathbb{S}_+^m$ and $\mathbb{S}_-^m$ denote the cones of positive and negative semidefinite matrices.
We denote by $\mathrm{tr}(Y)$ the trace of a matrix $Y \in \mathbb{R}^{m \times m}$.
The Frobenius inner product on $\mathbb{R}^{m\times n}$ is given by
\(
\langle Z,Y\rangle := \mathrm{tr}(Z^\top Y),
\)
and the associated Frobenius norm is defined as $\|Z\|:=\sqrt{\langle Z,Z\rangle}$.
For any $X\in\mathbb{S}^m$, its eigenvalues are denoted and ordered as
$
\lambda_1(X)\ge \lambda_2(X)\ge \cdots \ge \lambda_m(X),
$
and its spectral decomposition is written as
\(
X = U \Lambda U^\top,
\)
where $U\in\mathbb{O}^m$ and $\Lambda=\mathrm{diag}(\lambda_1(X),\dots,\lambda_m(X))$.
We denote by $[X]_+$ the projection of $X$ onto the positive semidefinite cone $\mathbb{S}^m_+$. 
If $X = U \Lambda U^\top$ is a spectral decomposition, then
$
[X]_+ = U\,\mathrm{diag}([\lambda_1(X)]_+,\dots,[\lambda_m(X)]_+)\,U^\top,
$
where $[t]_+ := \max\{t,0\}$, for $t\in \mathbb{R}$.
Let $S\colon\mathbb{R}^n\to\mathbb{S}^m$ be a continuously differentiable matrix-valued mapping. 
Its derivative at $x\in\mathbb{R}^n$ is the linear operator $DS(x)\colon\mathbb{R}^n\to\mathbb{S}^m$ defined by
\(
DS(x)h := \sum_{i=1}^n S_i(x)h_i,
\)
where $S_i(x):=\partial S(x)/\partial x_i\in\mathbb{S}^m$. 
The adjoint operator $DS(x)^*\colon\mathbb{S}^m\to\mathbb{R}^n$ is given by
\[
DS(x)^*\Xi :=
\big(\langle S_1(x),\Xi\rangle, \ldots, \langle S_n(x),\Xi\rangle \big)^\top,
\qquad \forall \Xi \in \mathbb{S}^m.
\]
\section{Optimality Conditions for Nonlinear Semidefinite Programming}\label{sessao2}
Sequential optimality conditions play a fundamental role in nonlinear semidefinite programming, particularly in degenerate settings where classical constraint qualifications fail. This phenomenon is especially pronounced in problems involving complementarity structures, for which standard KKT conditions are no longer expected to hold at local minimizers. 
In this section, we revisit first-order optimality conditions for nonlinear semidefinite programming and emphasize the relevance of approximate stationarity concepts as genuine necessary conditions for optimality. 

Consider the nonlinear semidefinite programming problem (NLSDP)  given by:
\begin{equation}\label{NLSDP}
\begin{aligned}
\displaystyle\mathop{\rm minimize }_{x\in \mathbb{R}^n}\quad & f(x),\\
\text{subject to}\quad & H(x)\preceq 0,\\
& P(x)=0,
\end{aligned}
\tag{NLSDP}
\end{equation}
where $f\colon \mathbb{R}^n \to \mathbb{R}$, $H\colon \mathbb{R}^n \to \mathbb{S}^m$, and $P\colon \mathbb{R}^n \to \mathbb{R}^{p\times q}$ are continuously differentiable functions. 
The motivation for introducing matrix equality constraints in problem 
(\ref{NLSDP}) stems from the need to properly capture the structure of the tightened reformulation that will be introduced later. In this setting, spectral decompositions naturally induce equality conditions on specific matrix blocks, which are most conveniently expressed as matrix-valued equality constraints. 

Let us define the KKT conditions for (\ref{NLSDP}). We begin by defining the Lagrangian function of problem (\ref{NLSDP}),
\( 
L\colon\mathbb{R}^{n} \times  \mathbb{S}^{m} \times \mathbb{R}^{p\times q} \to \mathbb{R},
\)
given by:
\begin{equation*}
  L(x, \Omega,\Lambda) := f(x) +  \langle H(x), \Omega \rangle +\langle P(x), \Lambda \rangle.  
\end{equation*}
\begin{definition}
We say that the KKT conditions hold at a feasible point $\bar{x} \in \mathbb{R}^n$ for problem~\eqref{NLSDP} if there exists $(\Omega, \Lambda) \in \mathbb{S}^m_+ \times \mathbb{R}^{p\times q}$ such that
\begin{align*}
& \nabla f(\bar{x}) + DH(\bar{x})^* \Omega + DP(\bar{x})^* \Lambda = 0, \qquad \langle H(\bar{x}), \Omega \rangle = 0.
\end{align*}
In this case, $(\Omega, \Lambda)$ is called a pair of \emph{Lagrange multipliers} associated with $\bar{x}$.
\end{definition}
The condition $\langle H(\bar{x}), \Omega \rangle = 0$ is called the \emph{complementarity condition}. It can be equivalently rewritten as
\(
H(\bar{x})\Omega = 0\)
or equivalently \(
\lambda^{U}_i(H(\bar{x}))\, \lambda^{U}_i(\Omega) = 0, \) for all  \(i = 1, \dots, m,
\)
where \( H(\bar{x}) \) and \( \Omega \) are simultaneously diagonalizable by an orthogonal matrix \( U \). The notation \( \lambda_i^U(X) \) denotes the \(i\)-th eigenvalue of \(X\) in the basis induced by \(U\), that is,
\[
U^\top X U = \operatorname{diag}(\lambda_1^U(X), \dots, \lambda_m^U(X)).
\]
In particular, this ordering is determined by the common diagonalization and does not necessarily coincide with the usual nonincreasing ordering of eigenvalues.

Although the KKT conditions characterize potential candidates for local optimality, they do not, in general, hold as necessary conditions at every local minimizer. In order for these conditions to be valid as necessary optimality conditions, suitable \textit{constraint qualifications} (CQ) must be imposed.
\begin{definition}
Let \( \bar{x} \in \mathbb{R}^n \) be a feasible point for~\eqref{NLSDP},  \( r = \textnormal{rank}(H(\bar{x})) \), \( \{v_1, \dots, v_{m - r}\} \subset \mathbb{R}^m \) be an orthonormal basis for the null space of \( H(\bar{x}) \), \( P(\bar{x}) \in \mathbb{R}^{p \times q} \), and let \( \{s_1, \dots, s_p\} \subset \mathbb{R}^p \) and \( \{u_1, \dots, u_q\} \subset \mathbb{R}^q \) be orthonormal bases of \( \mathbb{R}^p \) and \( \mathbb{R}^q \), respectively.
We say that \( \bar{x} \) satisfies the \emph{nondegeneracy constraint qualification} if the set of vectors
\begin{equation*}
\left\{
DH(\bar{x})^*(v_i v_j^\top)
\right\}_{i,j=1}^{m-r}
\;\cup
\left\{ 
\left( s_i^\top P_1(\bar{x})  u_j, \dots, s_i^\top P_n(\bar{x})  u_j \right)
\,\middle|\, i = 1, \dots, p;\; j = 1, \dots, q 
\right\}
\end{equation*}
is linearly independent in \( \mathbb{R}^n \).
\end{definition}
The nondegeneracy condition plays a crucial role in the analysis of semidefinite programming problems, as it ensures the uniqueness of Lagrange multipliers associated with the semidefinite constraint. This property is fundamental for the stability of optimality conditions and is often used in the convergence analysis of algorithms.
\begin{definition}
\label{rob}
We say that a feasible point \( \bar{x} \in \mathbb{R}^n \) satisfies \emph{Robinson's condition} for problem~\eqref{NLSDP} if the operator
\(
DP(\bar{x}) \colon \mathbb{R}^n \to \mathbb{R}^{p \times q}
\)
is surjective, and there exists a direction \( d \in \mathbb{R}^n \) such that:
\(
DP(\bar{x})d = 0\) and \( H(\bar{x}) + DH(\bar{x})d \prec 0\).
\end{definition}
Robinson's condition is widely recognized for its role in ensuring the existence and boundedness of Lagrange multipliers. In the context of semidefinite programming, it ensures that the KKT conditions hold at a local minimizer and constitutes a regularity condition that is weaker than standard nondegeneracy.

Recent developments indicate that classical CQs, such as
Robinson’s condition, are not essential for establishing the convergence of
augmented Lagrangian–based methods. In particular, the sequential constant rank-type
conditions are naturally adapted to the behavior of iterative algorithms and do
not require boundedness or uniqueness of Lagrange multipliers; see, e.g.,~\cite{seqcq}. 

Despite the importance of classical CQs, such as
nondegeneracy and Robinson's condition, these assumptions may fail, especially
in degenerate contexts. This motivates the study of weaker, yet still meaningful
conditions that are satisfied by every local minimizer. In this setting, we now
introduce sequential optimality conditions for~\eqref{NLSDP}; see, e.g.,~\cite{daiana}.
\begin{definition}
We say that a feasible point $\bar{x} \in \mathbb{R}^n$ satisfies the Approximate-KKT (AKKT) condition for problem~\eqref{NLSDP} if  there exist sequences $\{x^k\} \subset \mathbb{R}^n$, with $x^k \to \bar{x}$, and $\{\Omega^k\} \subset \mathbb{S}_+^m$, $\{\Lambda^k\} \subset \mathbb{R}^{p\times q}$ such that:
\begin{align}
& \lim_{k \to \infty} \nabla f(x^k) + DH(x^k)^* \Omega^k + DP(x^k)^* \Lambda^k = 0,\label{opti}  \\
& \text{if } \lambda_i^{U}(H(\bar{x})) < 0, \text{ then } \lim_{k \to \infty} \lambda_i^{S_k}(\Omega^k) = 0, 
\end{align}
where $U$ and $S_k$ are orthogonal matrices that diagonalize $H(\bar{x})$ and $\Omega^k$, respectively, with $S_k \to U$ as $k \to \infty$.
\end{definition}

We emphasize that the convergence $S_k \to U$ guarantees that the eigenvalues of $H(\bar{x})$ and $\Omega^k$ are properly aligned, allowing us to identify active and inactive components in a way analogous to nonlinear programming. The notation $\lambda_i^{U}(H(\bar{x}))$ and $\lambda_i^{S_k}(\Omega^k)$ indicates that the eigenvalues are ordered according to the diagonalizations induced by $U$ and $S_k$, respectively. In other words, $$
H(\bar x) = U\,\mathrm{diag}(\lambda_1^U(X),\dots,\lambda_m^U(X))\,U^\top.$$

We now present a simple example to illustrate the use of sequential optimality conditions in NLSDP when classical KKT conditions fail at a local minimizer, see~\cite{daiana}.
\begin{example}
Consider
\[
\displaystyle\min_{x\in\mathbb R} \, 2x,
\quad\text{s.t.}\quad
H(x):=
\begin{pmatrix}
0 & x\\
x & -1
\end{pmatrix}\preceq 0 .
\]

\noindent\emph{(i) $\bar x=0$ is not a KKT point.}
The KKT conditions require the existence of some $\Omega \succeq 0$ such that
$
\nabla f(\bar x)+DH(\bar x)^*\Omega=0$ and $\langle H(\bar x),\Omega\rangle=0.
$
Write $\Omega=\begin{pmatrix}\mu_{11}&\mu_{12}\\ \mu_{12}&\mu_{22}\end{pmatrix}$. Since
$
DH(\bar x)^*\Omega= 2\mu_{12},$
stationarity gives $2+2\mu_{12}=0$. Hence, $\mu_{12}=-1$. Furthermore, $H(\bar x)=\begin{pmatrix}0&0\\0&-1\end{pmatrix}$, complementarity yields $\langle H(\bar x),\Omega\rangle = -\mu_{22}$, and hence, $\mu_{22}=0$.
Therefore, any candidate multiplier must be
\(
\Omega=\begin{pmatrix}\mu_{11}&-1\\ -1&0\end{pmatrix}.
\)
This contradicts $\Omega\succeq 0$,
so $\bar x$ is not a KKT point.

\noindent\emph{(ii)} $\bar{x}=0$ is an AKKT point: let $x^k := -\frac{k}{k^2-1}$ and 
$\Omega^k := \begin{pmatrix} k & -1 \\ -1 & 1/k \end{pmatrix} \in \mathbb{S}_+^2$. 
Then $x^k \to 0$ and $\nabla f(x^k)+DH(x^k)^*\Omega^k=0$.
Moreover, $H(x^k)$ and $\Omega^k$ commute. Hence, they admit a common orthogonal diagonalization. Let $S_k$ be an orthogonal matrix such that
$
S_k^\top H(x^k) S_k = \mathrm{diag}(\lambda_1^k,\lambda_2^k)$ and $
S_k^\top \Omega^k S_k = \mathrm{diag}(\mu_1^k,\mu_2^k)
$. The sequence $\{S_k\}$ may be chosen such that it converges to the identity matrix $I$, which diagonalizes $H(\bar{x})$.
Thus, $\lambda_2^I(H(\bar x))=-1<0$ and a straightforward computation gives $\lambda_2^{S_k}(\Omega^k)=0$. 
Therefore, $\bar{x}$ is a local minimizer that satisfies the AKKT condition, although it is not a KKT point.
\end{example}

In the next section, we investigate optimization problems with matrix complementarity constraints, focusing on their formulation, intrinsic structure, and associated optimality conditions. 

\section{NLSDP with Complementarity Constraints}\label{sessao3}
In this section, we focus on the analysis of optimization problems involving conic complementarity constraints, with particular emphasis on those defined over the cones of positive and negative semidefinite matrices. For this purpose, we consider a formulation that includes only complementarity constraints, temporarily setting aside classical equality and inequality constraints, as well as standard semidefinite constraints (positive or negative). It is worth noting that such constraints could be incorporated in a straightforward manner, following, for instance, the approach proposed in~\cite{leo}. In this paper, we study the \emph{semidefinite cone MPCC} (SDCMPCC), defined as
\begin{equation}\label{SDCMPCC}
				\begin{array}{cl}
					\tag{SDCMPCC}\displaystyle\mathop{\rm minimize }_{x\in \mathbb{R}^n}  & f(x),  \\
					{\rm subject\ to } & \langle G(x), H(x) \rangle \geq 0, \quad G(x) \succeq 0, \quad H(x)\preceq0,
				\end{array}
			\end{equation}
where $f\colon \mathbb{R}^n \to \mathbb{R}$ and $G,H\colon \mathbb{R}^n \to \mathbb{S}^m$ are continuously differentiable functions. We impose the inequality $\langle G(x), H(x) \rangle \geq 0$. Under $G(x) \succeq 0$ and $H(x) \preceq 0$, the spectral properties of the semidefinite cones yield $\langle G(x), H(x) \rangle \leq 0$. Hence, the constraints guarantee $\langle G(x), H(x) \rangle = 0$, which characterizes a \emph{complementarity constraint}. The inequality formulation is advantageous, as it naturally induces a sign condition on the associated Lagrange multiplier, which will be exploited in the algorithmic analysis.

In the following example, we illustrate how an SDCMPCC arises naturally in the context of bilevel optimization; see, for instance,~\cite{dempe2020bilevel, dempe2018optimality}.
\begin{example}
Consider the following optimistic bilevel programming problem:
\begin{equation*}
	\begin{aligned}
\mathop{\text{minimize}}\limits_{x,y} \quad & f(x,y), \\
	\text{subject to} \quad & g(x,y) \leq 0, \\
	& y \in \arg\min_{\bar{y}} \left\{ \varphi(x,\bar{y}) \;\middle|\; G(x,\bar{y}) \succeq 0 \right\},
	\end{aligned}
\end{equation*}
where $f, \varphi \colon \mathbb{R}^n \times \mathbb{R}^m \to \mathbb{R}$, $g\colon\mathbb{R}^n \times \mathbb{R}^m  \to \mathbb{R}^k$, and $G\colon\mathbb{R}^n\times \mathbb{R}^m \to \mathbb{S}^p$.
For a fixed $x \in \mathbb{R}^n$, the lower-level problem reduces to the following NLSDP:
\begin{equation*}
	\begin{aligned}
\mathop{\text{minimize}}\limits_{y} \quad & \varphi^x(y):=\varphi(x,y),\\
		\text{subject to} \quad & G^x(y):=G(x,y) \preceq 0.
	\end{aligned}
\end{equation*}

In bilevel optimization, a common strategy consists of reformulating the original problem as a single-level optimization problem. One such reformulation is obtained by replacing the lower-level problem with its optimality conditions. In particular, if these conditions are both necessary and sufficient, then the lower-level problem can be equivalently described by its KKT system. Thus, under suitable regularity assumptions (see, for instance, Assumption~4.1 in~\cite{dempe2018optimality}), we obtain the following single-level reformulation:
\[
\begin{aligned}
\min_{x,y,\Omega} \quad & f(x,y), \\
\text{s.t.} \quad 
& g(x,y) \le 0, \\
& \nabla \varphi^x(y) + DG^x(y)^* \Omega = 0, \\
& G(x,y) \preceq 0,\;\Omega \succeq 0,\;\langle G(x,y), \Omega \rangle \geq 0.
\end{aligned}
\]
This problem is precisely of the SDCMPCC type considered in this work, with additional standard equality and inequality constraints, which can be dealt with. For this reformulated problem, however, classical KKT conditions may fail to hold at local solutions, as discussed below.
\end{example}

It is well known that, in classical nonlinear semidefinite programming, Robinson’s condition guarantees that the KKT system provides necessary optimality conditions and that the set of Lagrange multipliers is nonempty and bounded. In contrast, problems involving semidefinite cone complementarity constraints exhibit difficulties analogous to those encountered in MPCC, particularly due to the systematic failure of standard CQs; see, e.g.,~\cite{dsun}. As a consequence, classical KKT conditions cannot, in general, be expected to hold at local minimizers.

This motivates the introduction of weaker stationarity concepts, such as C- and W-stationarity, which are better suited to capture the local optimality properties of solutions in the SDCMPCC setting; see, e.g.,~\cite{dsun,hector}.

To derive refined stationarity conditions tailored to~\eqref{SDCMPCC}, alternative formulations of the complementarity constraint have been proposed. Two prominent approaches are: (i) representing complementarity via a nonconvex cone constraint involving the graph of the normal cone to the positive semidefinite cone, $\operatorname{gph} N_{\mathbb{S}_+^n}$,
$(G(z), H(z)) \in \operatorname{gph} N_{\mathbb{S}_+^n},$ and (ii) expressing it as a nonsmooth equation based on the projection onto the positive semidefinite cone:
$$G(z) - \Pi_{\mathbb{S}_+^n}\bigl(G(z) + H(z)\bigr) = 0.$$

The former admits a geometric interpretation within variational analysis, enabling the use of normal cone calculus, whereas the latter allows the application of nonsmooth analysis techniques, leading to stationarity conditions based on generalized differentiation. Both perspectives contribute to a deeper understanding of the feasible set structure and the nature of optimal solutions.

Building on these reformulation-based approaches, works such as~\cite{dsun, yan2011smoothing} have introduced stationarity concepts for SDCMPCC. However, such notions are not necessarily aligned with the behavior of the algorithmic scheme considered in this work. In particular, they do not, in general, correspond to the stationarity conditions naturally generated by the augmented Lagrangian framework that we analyze. In this sense, our analysis focuses on stationarity notions that arise directly from the structure of the method and reflect the asymptotic behavior of the generated sequences.

For this reason, rather than relying on these formulations, we adopt an alternative approach based on a tightened reformulation of the problem. This strategy allows us to derive optimality conditions that are intrinsically connected to the structure of the algorithm.
More precisely, the tightened problem provides a natural framework to obtain W-stationarity as a first-order condition. This notion will serve as a fundamental building block for the development of stronger stationarity concepts, which are consistent with both the geometry of the problem and the asymptotic behavior of the sequences generated by our method.

The formal definition of the tightened problem is presented in the next section.
\section{Spectral Decomposition and Tightened Problem Formulation}\label{sessao4}
In this section, we introduce a local reformulation of the~\eqref{SDCMPCC} based on the spectral decomposition at a given feasible point. This construction leads to a locally equivalent problem, referred to as the \emph{tightened problem}, which isolates the biactive and inactive components of the complementarity constraints. The tightened formulation will play a central role in the derivation of optimality conditions and sequential stationarity concepts, following a structure analogous to that adopted in the MPCC literature.

Let $A \in \mathbb{S}^m$ and let
\(
A = U \Lambda(A) U^\top
\)
be a spectral decomposition of $A$, where $U \in \mathbb{O}^{m}$ and $\Lambda(A) = \operatorname{diag}(\lambda_1(A), \dots, \lambda_m(A))$. We define the index sets
$$
\alpha := \{ i : \lambda_i(A) > 0 \}, \quad
\beta := \{ i : \lambda_i(A) = 0 \}, \quad
\gamma := \{ i : \lambda_i(A) < 0 \}.
$$
Accordingly, $\Lambda(A)$ can be written in block-diagonal form as
\[
\Lambda(A) =
\begin{bmatrix}
\Lambda_{\alpha\alpha}(A) & 0 & 0 \\
0 & \Lambda_{\beta\beta}(A) & 0 \\
0 & 0 & \Lambda_{\gamma\gamma}(A)
\end{bmatrix},
\]
where $\Lambda_{\alpha\alpha}(A)$, $\Lambda_{\beta\beta}(A)$, and $\Lambda_{\gamma\gamma}(A)$ are diagonal matrices containing the eigenvalues indexed by $\alpha$, $\beta$, and $\gamma$, respectively, with dimensions $|\alpha|$, $|\beta|$, and $|\gamma|$.
More generally, given these index sets, any $B \in \mathbb{S}^{m}$ can be partitioned into the six off-diagonal blocks 
$B_{\alpha\beta}$, $B_{\alpha\gamma}$, and $B_{\beta\gamma}$
, together with the diagonal blocks $B_{\alpha\alpha}$, $B_{\beta\beta}$, and $B_{\gamma\gamma}$.
Applying this decomposition to the matrices $G(\bar{x})$ and $H(\bar{x})$ at a feasible point $\bar{x}$, we obtain
\small{\[
G(\bar{x}) = U \begin{bmatrix}
\Lambda_{\alpha\alpha}(G(\bar{x})) & 0 & 0 \\
0 & \Lambda_{\beta\beta}(G(\bar{x})) & 0 \\
0 & 0 & 0
\end{bmatrix} U^\top, \;
H(\bar{x}) = U \begin{bmatrix}
0 & 0 & 0 \\
0 & \Lambda_{\beta\beta}(H(\bar{x})) & 0 \\
0 & 0 & \Lambda_{\gamma\gamma}(H(\bar{x}))
\end{bmatrix} U^\top,
\]}\normalsize
where $\Lambda_{\alpha\alpha}(G(\bar{x})) \succ 0$, $\Lambda_{\gamma\gamma}(H(\bar{x})) \prec 0$, and $\Lambda_{\beta\beta}(G(\bar{x})) = \Lambda_{\beta\beta}(H(\bar{x})) = 0$ correspond to the potentially biactive components.
Now, let $U\in\mathbb O^m$ be an orthogonal matrix that simultaneously diagonalizes
$G(\bar x)$ and $H(\bar x)$, and denote
\[
\widetilde G(x):=U^\top G(x)U,
\qquad
\widetilde H(x):=U^\top H(x)U.
\]
Then, at $\bar x$, we have
\small{\[
(\widetilde G(\bar x))_{\alpha\alpha} \succ 0, \quad 
(\widetilde G(\bar x))_{\alpha\beta} = (\widetilde G(\bar x))_{\alpha\gamma} = 0, \quad 
(\widetilde G(\bar x))_{\beta\beta} = (\widetilde G(\bar x))_{\beta\gamma} = (\widetilde G(\bar x))_{\gamma\gamma} = 0,
\]
\[
(\widetilde H(\bar x))_{\gamma\gamma} \prec 0, \quad 
(\widetilde H(\bar x))_{\alpha\gamma} = (\widetilde H(\bar x))_{\beta\gamma} = 0, \quad 
(\widetilde H(\bar x))_{\alpha\alpha} = (\widetilde H(\bar x))_{\alpha\beta} = (\widetilde H(\bar x))_{\beta\beta} = 0.
\]}
\normalsize
Thus, we define the tightened problem associated with $\bar{x}$ as
\begin{equation}\label{TSDCMPCC_refined}
\begin{array}{cl}\tag{TNLP($\bar{x}$)}
\displaystyle\mathop{\rm minimize }_{x\in \mathbb{R}^n}   & f(x), \\\\
\text{subject to} 
& (\widetilde G(x))_{\beta\beta} = 0, \quad (\widetilde G(x))_{\gamma\gamma} = 0, \\
& (\widetilde G(x))_{\alpha\beta} = 0, \quad (\widetilde G(x))_{\alpha\gamma} = 0, \quad (\widetilde G(x))_{\beta\gamma} = 0, \\
& (\widetilde H(x))_{\beta\beta} = 0, \quad (\widetilde H(x))_{\alpha\alpha} = 0, \quad (\widetilde H(x))_{\alpha\beta} = 0, \\
& (\widetilde H(x))_{\alpha\gamma} = 0, \quad (\widetilde H(x))_{\beta\gamma} = 0, \\
& (\widetilde G(x))_{\alpha\alpha} \succeq 0, \quad (\widetilde H(x))_{\gamma\gamma} \preceq 0.
\end{array}
\end{equation}

Note that the tightened problem~\eqref{TSDCMPCC_refined} depends explicitly on the reference point $\bar{x}$, which is itself feasible for~\eqref{TSDCMPCC_refined}. Moreover, the feasible set of~\eqref{TSDCMPCC_refined} is locally contained in the feasible set of~\eqref{SDCMPCC}, then $\bar{x}$ is also a local minimizer of~\eqref{TSDCMPCC_refined}. This reformulation provides a framework for defining constraint qualifications that can be used to derive optimality conditions for the original problem~\eqref{SDCMPCC} by applying the usual conditions for the standard problem~\eqref{TSDCMPCC_refined}. Recently, in~\cite{hector}, a tightened formulation was presented for general conic complementarity constraints, which reduces to~\eqref{TSDCMPCC_refined} in the case of the positive-semidefinite cone.

We now introduce a CQ for~\eqref{SDCMPCC}, defined in terms of the block structure induced by the tightened reformulation. This condition is a natural extension of the classical Robinson's condition from nonlinear semidefinite programming to the present complementarity setting.
\begin{definition}
\begin{sloppypar}
We say that a feasible point \( \bar{x} \in \mathbb{R}^n \) satisfies the \emph{SDCMPCC-Robinson's condition} if the linear operator
\end{sloppypar}
\small{\[
\begin{aligned}
\mathcal{A} := 
\big(& 
D\big((\widetilde G)_{\beta\beta}\big)(\bar{x}),\,
D\big((\widetilde G)_{\gamma\gamma}\big)(\bar{x}),\,
D\big((\widetilde G)_{\alpha\beta}\big)(\bar{x}),\,
D\big((\widetilde G)_{\alpha\gamma}\big)(\bar{x}),\,
D\big((\widetilde G)_{\beta\gamma}\big)(\bar{x}), \\
& D\big((\widetilde H)_{\beta\beta}\big)(\bar{x}),\,
D\big((\widetilde H)_{\alpha\alpha}\big)(\bar{x}),\,
D\big((\widetilde H)_{\alpha\beta}\big)(\bar{x}),\,
D\big((\widetilde H)_{\alpha\gamma}\big)(\bar{x}),\,
D\big((\widetilde H)_{\beta\gamma}\big)(\bar{x}) 
\big)
\colon \mathbb{R}^n \to \mathcal{S}.
\end{aligned}
\]}\normalsize is surjective, where
\[
\mathcal{S} = \mathbb{S}^{|\beta|} \times \mathbb{S}^{|\gamma|} \times \mathbb{R}^{|\alpha|\times|\beta|} \times \mathbb{R}^{|\alpha|\times|\gamma|} \times \mathbb{R}^{|\beta|\times|\gamma|}
\times \mathbb{S}^{|\beta|} \times \mathbb{S}^{|\alpha|} \times \mathbb{R}^{|\alpha|\times|\beta|} \times \mathbb{R}^{|\alpha|\times|\gamma|} \times \mathbb{R}^{|\beta|\times|\gamma|}.
\]
\end{definition}

 \begin{remark}\begin{sloppypar}
The SDCMPCC-Robinson condition coincides with the classical Robinson's condition applied to the tightened problem~\eqref{TSDCMPCC_refined}. This connection is particularly relevant, as it allows us to employ standard CQs on the tightened formulation in order to derive optimality conditions for the original problem~\eqref{SDCMPCC}. In this setting,  SDCMPCC-Robinson coincides with SDCMPCC-nondegeneracy, as defined in~\cite{dsun}. Indeed, the tightened formulation consists only of equality constraints, while the semidefinite constraints are strictly satisfied at $\bar{x}$. As a consequence, no interiority condition needs to be verified, that is, one can take $d=0$ in Definition~\ref{rob}. Although these conditions are usually treated as distinct in the literature, this distinction arises from the presence of additional constraints, which are absent in the present formulation.
\end{sloppypar}\end{remark}

We now introduce the KKT conditions associated with~\eqref{TSDCMPCC_refined}.
\begin{definition}
Let \( \bar{x} \in \mathbb{R}^n \) be a feasible point of~\eqref{TSDCMPCC_refined}. Let \(U \in \mathbb{O}^m\) associated with the spectral decomposition at \(\bar{x}\) used to define the tightened formulation, and define
\[
\widetilde{G}(x):=U^\top G(x)U, \qquad \widetilde{H}(x):=U^\top H(x)U,
\]
together with the corresponding index sets \(\alpha,\beta,\gamma\) and block partitions.
We say that \( \bar{x} \) satisfies the \emph{KKT conditions} for the tightened problem if there exist symmetric matrices
\(
\Gamma_{G}, \Gamma_{H} \in \mathbb{S}^m
\)
such that
\begin{equation}\label{kkt_tight_stationarity2}
\nabla f(\bar{x}) + DG(\bar{x})^\ast \Gamma_{G} + DH(\bar{x})^\ast \Gamma_{H} = 0,
\end{equation}
and, writing
\begin{equation}\label{multiplicadores1}
\Gamma_{\widetilde{G}}:=U^\top\Gamma_{G} U, \qquad \Gamma_{\widetilde{H}}:=U^\top\Gamma_{H} U,
\end{equation}
the  matrices \(\Gamma_{\tilde{G}}\) and \(\Gamma_{\widetilde{H}}\) satisfy $(\Gamma_{\widetilde{G}})_{\alpha\alpha} =0$ and $
(\Gamma_{\tilde{H}})_{\gamma\gamma} =0.$
\end{definition}
\section{Stationarity Notions}\label{sessao5}

In this section, we introduce and formalize notions of exact and sequential stationarity tailored to~\eqref{SDCMPCC}. Our approach builds on the framework developed in~\cite{leo} for MPCC, which we extend to accommodate the structure of~\eqref{SDCMPCC} through the tightened reformulation introduced in the previous section. 

The exact notions characterize optimality at a candidate solution in terms of first-order conditions for the tightened problem, while the sequential notions capture the behavior of approximate solutions generated by iterative methods. In particular, W-stationarity corresponds to first-order optimality of the tightened problem, whereas AW-stationarity represents its sequential counterpart. The notions of C- and AC-stationarity strengthen these conditions by imposing a compatibility requirement on the biactive block of the multipliers.

This unified treatment provides a framework that is both adapted to the complementarity structure of the problem and directly aligned with the convergence properties of the augmented Lagrangian method analyzed later in the paper.
\subsection{Exact Stationarity Notions}
We now introduce the notions of stationarity that will be used in the sequel. 
Since our approach is based on the tightened reformulation introduced in the previous section, the relevant block structure is the one induced by the fixed orthogonal matrix \(U\) and by the partition \((\alpha,\beta,\gamma)\) associated with the reference feasible point \(\bar x\).

\begin{definition}
We say that a feasible point \(\bar{x}\) is a \emph{weakly stationary} \emph{(W-statio-nary)} point for~\eqref{SDCMPCC} if it satisfies the KKT conditions of the tightened problem~\eqref{TSDCMPCC_refined} associated with \(\bar{x}\).
\end{definition}

W-stationarity is the weakest stationarity notion in our framework. It only requires the existence of multipliers satisfying the first-order optimality conditions of the tightened problem, without imposing any additional sign or compatibility condition on the biactive block.

\begin{definition}
Let \(\bar{x}\) be a W-stationary point of~\eqref{SDCMPCC} and let \(\Gamma_G,\Gamma_H\in\mathbb S^m\) be associated multipliers. We say that \(\bar{x}\) is a \emph{C-stationary point} if
\[
\left\langle
(\widetilde\Gamma_G)_{\beta\beta},
(\widetilde\Gamma_H)_{\beta\beta}
\right\rangle \le 0,
\]
where
\[
\widetilde\Gamma_G:=U^\top\Gamma_G U,
\qquad
\widetilde\Gamma_H:=U^\top\Gamma_H U.
\]
\end{definition}

The above notion of C-stationarity is closely related to the definition proposed in~\cite{dsun}. In
our setting, it arises naturally as a compatibility condition on the multipliers associated with
the biactive block, reflecting the complementarity structure of the original problem within
the tightened formulation. While the approach in~\cite{dsun} is based on variational analysis and
normal cone characterizations of the complementarity set, leading to coupling conditions
involving both diagonal and off-diagonal blocks of the multipliers, our construction is derived
from a tightened reformulation induced by the spectral structure, in which off-diagonal
components are explicitly eliminated through equality constraints. As a consequence, the
resulting stationarity condition does not involve cross-block coupling terms and reduces to a
compatibility condition restricted to the biactive block. Despite these differences, both
frameworks rely on the same underlying partition of the variables and capture analogous
first-order optimality features. A precise equivalence between the two notions is not pursued
here and remains an open question.

At this stage, W-stationarity can be obtained as a first-order optimality condition under standard constraint qualifications for the tightened problem, such as Robinson’s condition. Stronger stationarity properties will be derived in subsequent sections.
\subsection{Approximate Stationarity Notions}
We now introduce sequential stationarity concepts for the problem~\eqref{SDCMPCC}, extending the Approximate Karush–
Kuhn–Tucker (AKKT) condition and the notions of AW- and AC-stationarity from MPCC~\cite{leo} to the semidefinite setting.
These notions are motivated by the fact that, in the presence of complementarity constraints, standard constraint qualifications typically fail, so classical KKT conditions are not guaranteed to hold at local minimizers. In contrast, sequential conditions are genuine necessary optimality conditions, satisfied by every local minimizer without any constraint qualification. Moreover, they are consistent with the behavior of iterative methods, as they naturally arise from algorithmic stopping criteria.
We begin with the definition of AW-stationarity.
\begin{definition}\label{aw}
Let $\bar x$ be a feasible point of~\eqref{SDCMPCC}, and let
$U \in \mathbb{O}^m$ be an orthogonal matrix that simultaneously diagonalizes
$G(\bar x)$ and $H(\bar x)$.
We say that $\bar x$ satisfies the \emph{AW}-stationarity condition for~\eqref{SDCMPCC} if
there exist sequences $\{x^k\}\subset \mathbb R^n$, $\{\Gamma_{G}^k\}\subset \mathbb S^m$,
$\{\Gamma_{H}^k\}\subset \mathbb S^m$, and sequences $\{S_k\},\{V_k\}\subset \mathbb O^m$
such that $S_k$ diagonalizes $G(x^k)$, $V_k$ diagonalizes $H(x^k)$,
\[
x^k\to\bar x,\qquad S_k\to U,\qquad V_k\to U,
\]
and
\begin{equation}\label{a}
\nabla f(x^k)+DG(x^k)^*\,\Gamma_G^{k}+DH(x^k)^*\,\Gamma_H^{k}\to0,
\end{equation}
together with
\begin{equation}\label{b}
(\widetilde\Gamma_{G}^k)_{\alpha\alpha}
:=
(S_k^\top \Gamma_{G}^k S_k)_{\alpha\alpha}\to0,
\qquad
(\widetilde\Gamma_{H}^k)_{\gamma\gamma}
:=
(V_k^\top \Gamma_{H}^k V_k)_{\gamma\gamma}\to0.
\end{equation}
\end{definition}
Note that AW-stationarity coincides with the AKKT condition applied to the tightened problem \textnormal{TNLP($\bar{x}$)}. 
Consequently, AW-stationarity constitutes a sequential optimality condition for the tightened problem \textnormal{TNLP($\bar{x}$)}, and therefore also provides a necessary optimality condition for the original problem~\eqref{SDCMPCC}.

We now introduce a stationarity concept tailored for the SDCMPCC framework, namely 
\emph{approximate Clarke stationarity} (AC-stationarity).
\begin{definition}\label{def:AC}
Let $\bar{x}$ be an AW-stationary point of~\eqref{SDCMPCC}, and let 
$\{x^k\}$, $\{\Gamma_G^k\}$, $\{\Gamma_H^k\}$, $\{S_k\}$, and $\{V_k\}$
be sequences satisfying Definition~\ref{aw}.
We say that $\bar{x}$ satisfies the \emph{approximate Clarke stationarity} 
(AC-stationarity) condition if
\begin{equation}\label{AC}
\limsup_{k\to\infty}
\left\langle 
(\widetilde\Gamma_G^k)_{\beta\beta},\,
(\widetilde\Gamma_H^k)_{\beta\beta}
\right\rangle
\le 0,
\end{equation}
where
\[
\widetilde\Gamma_G^k := S_k^\top \Gamma_G^k S_k,
\qquad
\widetilde\Gamma_H^k := V_k^\top \Gamma_H^k V_k.
\]
\end{definition}
It is straightforward to verify that every AC-stationary point is also AW-stationary. The converse, however, does not hold in general. 
\begin{example}
Consider
\[
\min_{x\in\mathbb R^6} f(x)
\quad\text{s.t.}\quad
G(x)\succeq0,\quad H(x)\preceq0,\quad \langle G(x),H(x)\rangle\ge0,
\]
where
\(
f(x)=-x_1-x_3+x_4+x_6,
\)
\(
G(x)=
\begin{pmatrix}
x_1 & x_2\\
x_2 & x_3
\end{pmatrix}\) and
\(H(x)=-
\begin{pmatrix}
x_4 & x_5\\
x_5 & x_6
\end{pmatrix}.
\)
Let \(\bar x=0\). Then \(G(\bar x)=H(\bar x)=0\), hence
\(
I_\alpha(\bar x)=I_\gamma(\bar x)=\emptyset\) and 
\(
I_\beta(\bar x)=\{1,2\}.
\)
We first demonstrate that \(\bar x\) is AW-stationary. Take \(x^k=\bar x\),
\(S_k=V_k=I\), and
\(
\Gamma_G^k=\Gamma_H^k=I.
\)
Since
\[
\nabla f(\bar x)=(-1,0,-1,1,0,1),
\]
and
\[
DG(\bar x)^*I=(1,0,1,0,0,0),
\qquad
DH(\bar x)^*I=(0,0,0,-1,0,-1),
\]
we obtain
\[
\nabla f(\bar x)+DG(\bar x)^*\Gamma_G^k
+DH(\bar x)^*\Gamma_H^k=0.
\]
Thus, \(\bar x\) is AW-stationary.

We now prove that \(\bar x\) is not AC-stationary. Suppose, by contradiction, that
\(\bar x\) satisfies AC-stationarity. Then there exist sequences satisfying
\[
\nabla f(x^k)+DG(x^k)^*\Gamma_G^k
+DH(x^k)^*\Gamma_H^k\to0.
\]
Since \(G\) and \(H\) are linear, this condition gives componentwise
\[
(\Gamma_G^k)_{11}\to1,\qquad
(\Gamma_G^k)_{12}\to0,\qquad
(\Gamma_G^k)_{22}\to1,
\]
and
\[
(\Gamma_H^k)_{11}\to1,\qquad
(\Gamma_H^k)_{12}\to0,\qquad
(\Gamma_H^k)_{22}\to1.
\]
Therefore,
\(
\Gamma_G^k\to I\) and 
\(
\Gamma_H^k\to I.
\)
Since the whole block is biactive, we have
\[
(\widetilde\Gamma_G^k)_{\beta\beta}\to I,
\qquad
(\widetilde\Gamma_H^k)_{\beta\beta}\to I.
\]
Consequently,
\[
\lim_{k\to\infty}
\left\langle
(\widetilde\Gamma_G^k)_{\beta\beta},
(\widetilde\Gamma_H^k)_{\beta\beta}
\right\rangle
=
\langle I,I\rangle
=
2>0,
\]
which contradicts the AC-stationarity condition. Hence, \(\bar x\) is AW-stationary
but not AC-stationary.
\end{example}

To prove that AC-stationarity is a sequential optimality condition for the problem~\eqref{SDCMPCC}, 
we follow the strategy of~\cite{leo} and introduce an auxiliary reformulation of the problem. 
This reformulation decouples the complementarity structure from \( G(x) \) and \( H(x) \), 
allowing the analysis of the resulting sequential stationarity property.
We reformulate~\eqref{SDCMPCC} by introducing auxiliary symmetric matrix variables $W_G, W_H \in \mathbb{S}^m$:
\begin{equation}\label{SDCMPCC-R}
\tag{SDCMPCC-R} 
\begin{aligned}
\displaystyle\mathop{\rm minimize }_{x,\, W_G,\, W_H} \quad & f(x), \\
\text{subject to} \quad
&W_G = G(x),  W_H = H(x), \\
& W=(W_G,W_H) \in \mathcal{W},  
\end{aligned}
\end{equation}
where $\mathcal{W} := \left\{ (W_G, W_H) \in \mathbb{S}^m_{+} \times \mathbb{S}^m_{-} \;\middle|\;
 \langle W_G, W_H \rangle \geq 0 \right\}.$ 
 The validity of Guignard’s constraint qualification for the feasible set $\mathcal{W}$ is established in Appendix~A (Theorem~\ref{guinard}).
We are now in a position to establish a fundamental result: every local minimizer of~\eqref{SDCMPCC} admits sequences of points and symmetric multipliers that approximate stationarity, where the sequence of points converges to a feasible point. 
\begin{theorem}
\label{thm:AKKT-SDCMPCC}
Let $\bar{x}$ be a local minimizer of~\eqref{SDCMPCC}. Then there exist sequences satisfying the AC-stationarity conditions.
\end{theorem}

\begin{proof}
Let $\delta>0$ be such that $f(\bar x)\leq f(x)$ for all $x$ feasible to~\eqref{SDCMPCC} with $\|x-\bar{x}\|\leq\delta$. For a penalty parameter $\rho^k\to\infty$, consider the bounded, penalized problems
\begin{equation}\label{eq:pen}
\begin{aligned}
\min_{x,W_G,W_H}\;& f(x)+\tfrac12\|x-\bar x\|^2
+\tfrac{\rho^k}{2}\!\left(\|W_G-G(x)\|^2+\|W_H-H(x)\|^2\right),\\
\text{s.t.}\;& (W_G,W_H)\in\mathcal W,\quad
\|x-\bar x\|\le\delta,\ \|W_G-G(\bar x)\|\le\delta,\ \|W_H-H(\bar x)\|\le\delta.
\end{aligned}
\end{equation}
The feasible set is nonempty and compact. Let $(x^k,W_G^k,W_H^k)$ be a global minimizer
of~\eqref{eq:pen}. From the optimality of $(x^k, W_G^k, W_H^k)$, the fact that $\rho^k \to \infty$, and the boundedness of the objective function in~\eqref{eq:pen} by $f(\bar{x})$, we have
\begin{equation}\label{eq:cons}
\|W_G^k-G(x^k)\|\to 0,\qquad \|W_H^k-H(x^k)\|\to 0,
\end{equation}
and every limit point $(\hat x,\hat W_G,\hat W_H)$ satisfies $\hat W_G=G(\hat x)$, $\hat W_H=H(\hat x)$ and
$(\hat W_G,\hat W_H)\in\mathcal W$, because $\mathcal{W}$ is closed. From the definition of $\delta$, we have $f(\bar x)\leq f(\hat{x})$. Now, for each $k$, the optimal value of~\eqref{eq:pen} is clearly bounded from below by $f(x^k)+\frac{1}{2}\|x^k-\bar{x}\|^2$ and from above by $f(\bar{x})\leq f(\hat x)$. 
Taking limits, we obtain $f(\hat x)+\frac{1}{2}\|\hat x-\bar x\|^2\leq f(\hat x)$, which implies that $\hat{x} = \bar{x}$.

Therefore, for $k$ large enough, the only locally relevant constraints in~\eqref{eq:pen} are $(W_G^k,W_H^k)\in\mathcal W$. By Guignard's CQ, the KKT conditions hold. Hence, there exist multipliers
\[
\theta_G^k\in\mathbb S^m_-,
\qquad
\theta_H^k\in\mathbb S^m_+,
\qquad
\gamma^k\in\mathbb R_+
\]
such that
\begin{equation}\label{eq:kktx}
\nabla f(x^k)+(x^k-\bar x)
-\rho^kDG(x^k)^*(W_G^k-G(x^k))
-\rho^kDH(x^k)^*(W_H^k-H(x^k))
=0,
\end{equation}
\begin{equation}\label{eq:kktw}
\rho^k(W_G^k-G(x^k))-\gamma^k W_H^k+\theta_G^k=0,
\qquad
\rho^k(W_H^k-H(x^k))-\gamma^k W_G^k+\theta_H^k=0,
\end{equation}
and
\begin{equation}\label{eq:kktcomp}
\langle \theta_G^k,W_G^k\rangle=0,\qquad
\langle \theta_H^k,W_H^k\rangle=0,\qquad
\gamma^k\langle W_G^k,W_H^k\rangle=0.
\end{equation}
Define
\[
\Gamma_G^k:=-\rho^k(W_G^k-G(x^k)),
\qquad
\Gamma_H^k:=-\rho^k(W_H^k-H(x^k)).
\]
Then~\eqref{eq:kktw} gives
\begin{equation}\label{eq:gammarel}
\Gamma_G^k=\theta_G^k-\gamma^kW_H^k,
\qquad
\Gamma_H^k=\theta_H^k-\gamma^kW_G^k.
\end{equation}
Moreover, from~\eqref{eq:kktx},
\[
\nabla f(x^k)+DG(x^k)^*\Gamma_G^k
+DH(x^k)^*\Gamma_H^k
=-(x^k-\bar x)\to0.
\]

Now, let $P_k\in\mathbb O^m$ simultaneously
diagonalize $W_G^k$ and $W_H^k$, and take a subsequence such that $P_k\to U$,
where $U$ simultaneously diagonalizes $G(\bar x)$ and $H(\bar x)$. For any matrix
$X^k$, write
\(
\widetilde X^k:=P_k^\top X^kP_k.
\)
Let $(\alpha,\beta,\gamma)$ be the spectral partition associated with
$G(\bar x)$ and $H(\bar x)$.
Since
\[
\widetilde W_G^k\to U^\top G(\bar x)U,
\qquad
\widetilde W_H^k\to U^\top H(\bar x)U,
\]
we have, for all large $k$,
\(
(\widetilde W_G^k)_{\alpha\alpha}\succ0\),
\(
(\widetilde W_H^k)_{\gamma\gamma}\prec0,
\)
while the corresponding complementary blocks vanish according to the structure of
$\mathcal W$.
Since $(\widetilde W_G^k)_{\alpha\alpha}\succ0$ and $\langle \theta_G^k,W_G^k\rangle=0$, we obtain
\(
(\widetilde\theta_G^k)_{\alpha\alpha}=0\).
Moreover, because $W_H^k\preceq0$ and $W_G^k\succeq0$ with
$\langle W_G^k,W_H^k\rangle=0$, we also have
\(
(\widetilde W_H^k)_{\alpha\alpha}=0\).
Using~\eqref{eq:gammarel}, we conclude that
\(
(\widetilde\Gamma_G^k)_{\alpha\alpha}=0\).
Analogously, from $\theta_H^k\succeq0$, $W_H^k\preceq0$, and
$\langle \theta_H^k,W_H^k\rangle=0$, we obtain
\(
(\widetilde\theta_H^k)_{\gamma\gamma}=0\)
Since the corresponding blocks of $\widetilde W_G^k$ vanish on the $\gamma$-part,
\eqref{eq:gammarel} yields
\(
(\widetilde\Gamma_H^k)_{\gamma\gamma}=0\).

It remains to verify the condition on the biactive block. From~\eqref{eq:gammarel},
\[
(\widetilde\Gamma_G^k)_{\beta\beta}
=
(\widetilde\theta_G^k)_{\beta\beta}
-\gamma^k(\widetilde W_H^k)_{\beta\beta},
\qquad
(\widetilde\Gamma_H^k)_{\beta\beta}
=
(\widetilde\theta_H^k)_{\beta\beta}
-\gamma^k(\widetilde W_G^k)_{\beta\beta}.
\]
Taking inner products, we obtain
\small{\[
\begin{aligned}
\left\langle
(\widetilde\Gamma_G^k)_{\beta\beta},
(\widetilde\Gamma_H^k)_{\beta\beta}
\right\rangle
&=
\left\langle
(\widetilde\theta_G^k)_{\beta\beta},
(\widetilde\theta_H^k)_{\beta\beta}
\right\rangle 
-\gamma^k
\left\langle
(\widetilde\theta_G^k)_{\beta\beta},
(\widetilde W_G^k)_{\beta\beta}
\right\rangle
-\gamma^k
\left\langle
(\widetilde W_H^k)_{\beta\beta},
(\widetilde\theta_H^k)_{\beta\beta}
\right\rangle \\
&\quad
+(\gamma^k)^2
\left\langle
(\widetilde W_H^k)_{\beta\beta},
(\widetilde W_G^k)_{\beta\beta}
\right\rangle.
\end{aligned}
\]}\normalsize By~\eqref{eq:kktcomp} and the sign properties of the multipliers and variables,
all terms except the first vanish, and
\[
\left\langle
(\widetilde\Gamma_G^k)_{\beta\beta},
(\widetilde\Gamma_H^k)_{\beta\beta}
\right\rangle =\left\langle
(\widetilde\theta_G^k)_{\beta\beta},
(\widetilde\theta_H^k)_{\beta\beta}
\right\rangle\le 0.
\]
Hence,
\[
\limsup_{k\to\infty}
\left\langle
(\widetilde\Gamma_G^k)_{\beta\beta},
(\widetilde\Gamma_H^k)_{\beta\beta}
\right\rangle \le 0.
\]
This proves the AC-stationarity condition.
\end{proof}
\begin{remark}
The sequential stationarity notion introduced above yields a necessary optimality condition for the original problem~\eqref{SDCMPCC} without requiring any constraint qualification. This feature is particularly relevant in the present setting, since standard constraint qualifications typically fail in the presence of complementarity constraints.
Moreover, as will be established later, AC-stationarity serves as an intermediate tool in the derivation of stronger stationarity concepts: under a suitable MPCC-type constraint qualification, it implies C-stationarity.
\end{remark}
\section{An Augmented Lagrangian Strategy Based on a Slack Reformulation}\label{algorithm}
In this section, we investigate the algorithmic consequences of the stationarity framework developed in the previous sections. In particular, we analyze the behavior of augmented Lagrangian methods when applied to the problem~\eqref{SDCMPCC} and establish connections between the generated sequences and the sequential stationarity concepts introduced earlier. Our goal is to prove that, under suitable assumptions, accumulation points of sequences produced by an augmented Lagrangian strategy satisfy meaningful first-order optimality conditions in the sense of C-stationarity. These results provide a theoretical justification for the use of augmented Lagrangian schemes in the presence of semidefinite cone complementarity constraints and severe degeneracy.

The main idea is to work with a slack-variable reformulation, similarly to what was done in Subsection~5.2, while also incorporating the strategy proposed in~\cite{AndreaniHaeser2025LowerLevel}. In particular, instead of penalizing all constraints, we keep the most difficult ones explicitly in the subproblems, following the lower-level framework introduced in~\cite{AndreaniHaeser2025LowerLevel}. This approach allows the algorithm to preserve feasibility concerning these constraints throughout the iterations, leading to improved numerical stability and better theoretical properties, such as the boundedness of the dual sequence and enhanced convergence behavior. 

Recall that problem~\eqref{SDCMPCC} can be equivalently written in the
following slack-variable form:
\[
\begin{array}{ll}
\min & f(x),\\[0.1cm]
\text{s.t.} & W_G=G(x),\\
            & W_H=H(x),\\
            & (W_G,W_H)\in\mathcal W,
\end{array}
\]
where
\(
\mathcal W:=
\{(W_G,W_H)\in\mathbb S^m_+\times\mathbb S^m_-:
\langle W_G,W_H\rangle\geq 0\}.
\)
Thus, the semidefinite complementarity structure is kept explicitly in the
subproblems, while the equalities linking the slack variables to the original
mappings are handled by an augmented Lagrangian term.

Given $\rho>0$ and multiplier estimates
$(\bar\Lambda_G,\bar\Lambda_H)\in\mathbb S^m\times\mathbb S^m$,
we define the augmented Lagrangian function as follows:
\small{
\[
\begin{aligned}
\mathcal{L}_\rho(x,W_G,W_H,\bar\Lambda_G,\bar\Lambda_H):={}
f(x)
+
\frac{\rho}{2}
\left\|
W_G-G(x)+\frac{\bar\Lambda_G}{\rho}
\right\|^2
+
\frac{\rho}{2}
\left\|
W_H-H(x)+\frac{\bar\Lambda_H}{\rho}
\right\|^2.
\end{aligned}
\]}\normalsize
No penalty term is associated with the complementarity constraint itself.
Instead, the condition $(W_G,W_H)\in\mathcal W$ is imposed directly in each
subproblem.
Let $\theta_G\in\mathbb S_-^m$, $\theta_H\in\mathbb S_+^m$, and $\gamma\geq0$
denote the multipliers associated with
\[
W_G\succeq0,\qquad W_H\preceq0,\qquad
\langle W_G,W_H\rangle\geq0.
\]
We now present a formal description of the proposed algorithm.

\begin{algorithm}[H]
\caption{Augmented Lagrangian Algorithm for SDCMPCC}
\label{alg:aug_lagrangian_sdcmpcc_slack}
\begin{algorithmic}[1]

\item[Step 1:] \textbf{(Initialization):}
Choose $\rho_1>0$, $\eta>1$, $\tau\in(0,1)$, and sequences
$\{\varepsilon_k\}\subset\mathbb R_+$ and $\{\delta_k\}\subset\mathbb R_+$
such that
\(
\varepsilon_k\to0\) and 
\(
\delta_k\to0.
\)
Let $\mathcal B_G,\mathcal B_H\subset\mathbb S^m$ be bounded convex sets.
Choose $(\bar\Lambda_G^1,\bar\Lambda_H^1)\in\mathcal B_G\times\mathcal B_H$ and
$(x^0,W_G^0,W_H^0)$. Set $k:=1$.

\item[Step 2:] \textbf{(Subproblem):}
Compute 
\((x^k,W_G^k,W_H^k)\), with \((W_G^k,W_H^k)\in\mathcal W\), and multipliers
\(\theta_G^k\in\mathbb S_-^m\), \(\theta_H^k\in\mathbb S_+^m\), \(\gamma_k\ge0\), such that
\[
\left\|
\begin{pmatrix}
 \nabla_x \mathcal{L}_{\rho_k}(x^k,W_G^k,W_H^k,\bar\Lambda_G^k,\bar\Lambda_H^k)\\
\nabla_{W_G}\mathcal L_{\rho_k}+\theta_G^k-\gamma_k W_H^k\\
\nabla_{W_H}\mathcal L_{\rho_k}+\theta_H^k-\gamma_k W_G^k
\end{pmatrix}
\right\|
\le \varepsilon_k,
\]
and
\[
\left|\langle \theta_G^k,W_G^k\rangle\right|
\le \delta_k,
\qquad
\left|\langle W_H^k,\theta_H^k\rangle\right|
\le \delta_k.
\]
\item[Step 3:] \textbf{(Penalty update):}
Define
\[
\mathcal V^k:=\max\{\|W_G^k-G(x^k)\|,\|W_H^k-H(x^k)\|\}.
\]
If $k>1$ and $\mathcal V^k\leq\tau\mathcal V^{k-1}$, set
$\rho_{k+1}:=\rho_k$; otherwise set $\rho_{k+1}:=\eta\rho_k$.

\item[Step 4:] \textbf{(Multiplier update):}
\[
\Lambda_G^{k+1}:=\bar\Lambda_G^k+\rho_k(W_G^k-G(x^k)),
\qquad
\Lambda_H^{k+1}:=\bar\Lambda_H^k+\rho_k(W_H^k-H(x^k)).
\]
Project:
\[
\bar\Lambda_G^{k+1}:=\operatorname{Proj}_{\mathcal B_G}(\Lambda_G^{k+1}),
\qquad
\bar\Lambda_H^{k+1}:=\operatorname{Proj}_{\mathcal B_H}(\Lambda_H^{k+1}).
\]

\item[Step 5:] 
Set $k:=k+1$ and return to Step~2.

\end{algorithmic}
\end{algorithm}
The following analysis shows that every accumulation point generated by the previous
Algorithm~\ref{alg:aug_lagrangian_sdcmpcc_slack} satisfies the AW-stationarity condition. 
In order to establish AC-stationarity, an additional assumption is required, which will be 
introduced below.

\begin{assumption}\label{ass:comp}
Let $\{(x^k,W_G^k,W_H^k,\theta_G^k,\theta_H^k,\gamma_k)\}$ be a sequence 
generated by Algorithm \ref{alg:aug_lagrangian_sdcmpcc_slack}. We assume that
\[
\delta_k \leq\frac{\varepsilon_k}{1+\gamma_k}.
\]
\end{assumption}
Assumption~\ref{ass:comp} is not restrictive. In fact, in augmented Lagrangian frameworks 
where the multipliers are defined via normal cone projections, such as in 
Jia, Kanzow, Mehlitz, and Wachsmuth~\cite{patrick}, these relations are automatically satisfied. 
Indeed, the complementarity terms follow directly from the normal cone structure, 
making the above condition hold trivially.
We now analyze the convergence properties of Algorithm~\ref{alg:aug_lagrangian_sdcmpcc_slack}. 
\begin{theorem}\label{teo_AC_alg}
Let 
\(
\{(x^k,W_G^k,W_H^k)\}
\subset \mathbb R^n\times\mathbb S^m\times\mathbb S^m
\)
be a sequence generated by Algorithm~\ref{alg:aug_lagrangian_sdcmpcc_slack}
applied to~\eqref{SDCMPCC}. Suppose that
\(
x^k\to\bar x,
\)
with $\bar x$ feasible for~\eqref{SDCMPCC}. Then $\bar x$ is an
AW-stationary point of~\eqref{SDCMPCC}.
Moreover, if Assumption~\ref{ass:comp} holds, then $\bar x$ is also
an AC-stationary point of~\eqref{SDCMPCC}.
\end{theorem}

\begin{proof}
By Step~2 of Algorithm~\ref{alg:aug_lagrangian_sdcmpcc_slack}, we have
\[
\nabla f(x^k)
+
DG(x^k)^*\Gamma_G^k
+
DH(x^k)^*\Gamma_H^k \to 0,
\]
and
\begin{equation}\label{v}
-\Gamma_G^k+\theta_G^k-\gamma^k W_H^k \to 0,
\qquad
-\Gamma_H^k+\theta_H^k-\gamma^k W_G^k \to 0,    
\end{equation}
where
\[
\Gamma_G^k:=
-\bar\Lambda_G^k-\rho_k(W_G^k-G(x^k)),
\qquad
\Gamma_H^k:=
-\bar\Lambda_H^k-\rho_k(W_H^k-H(x^k)).
\]
Thus, we obtain the approximate stationarity relation
\[
\nabla f(x^k)
+
DG(x^k)^*\Gamma_G^k
+
DH(x^k)^*\Gamma_H^k \to 0.
\]

Let $U$ be an orthogonal matrix that simultaneously diagonalizes
$G(\bar x)$ and $H(\bar x)$, and let $(\alpha,\beta,\gamma)$ be the associated
spectral partition. Let $U^k \to U$ be a sequence of orthogonal matrices that
simultaneously diagonalizes $W_G^k$ and $W_H^k$, and define
\[
\widetilde\Gamma_G^k:=(U^k)^\top\Gamma_G^kU^k,
\qquad
\widetilde\Gamma_H^k:=(U^k)^\top\Gamma_H^kU^k.
\]
From~\eqref{v}
and the complementarity conditions, we obtain the following properties.
On the $\alpha$-block, since  $(W_G^k)_{\alpha\alpha} \succ 0$ for k sufficiently large and
$\theta_G^k\preceq 0$ with
\(
\langle \theta_G^k,W_G^k\rangle=0,
\)
it follows that
\(
(\widetilde\theta_G^k)_{\alpha\alpha}=0\).
Moreover, since $(\widetilde W_H^k)_{\alpha\alpha}=0$, we conclude that
\(
(\widetilde\Gamma_G^k)_{\alpha\alpha} \to 0.
\)
Similarly, on the $\gamma$-block, we obtain
\(
(\widetilde\Gamma_H^k)_{\gamma\gamma} \to 0.
\)
Therefore, the AW-stationarity conditions are satisfied.

It remains to verify the approximate Clarke condition on the biactive block.
Since \(W_G^k\to G(\bar x)\) and \(W_H^k\to H(\bar x)\), we have
\(
(\widetilde W_G^k)_{\beta\beta}\to0
\)
and
\(
(\widetilde W_H^k)_{\beta\beta}\to0.
\)
Hence,
\[
(\widetilde\Gamma_G^k)_{\beta\beta}
=
(\widetilde\theta_G^k)_{\beta\beta}
-\gamma_k(\widetilde W_H^k)_{\beta\beta}
+ R_G^k,
\]
\[
(\widetilde\Gamma_H^k)_{\beta\beta}
=
(\widetilde\theta_H^k)_{\beta\beta}
-\gamma_k(\widetilde W_G^k)_{\beta\beta}
+ R_H^k,
\]
where \(R_G^k \to 0\) and \(R_H^k \to 0\).
We now analyze
\(
\left\langle
(\widetilde\Gamma_G^k)_{\beta\beta},
(\widetilde\Gamma_H^k)_{\beta\beta}
\right\rangle.
\)
Expanding, we obtain
\[
\begin{aligned}
\left\langle
(\widetilde\Gamma_G^k)_{\beta\beta},
(\widetilde\Gamma_H^k)_{\beta\beta}
\right\rangle
&=
\left\langle
(\widetilde\theta_G^k)_{\beta\beta},
(\widetilde\theta_H^k)_{\beta\beta}
\right\rangle \\
&\quad
-\gamma_k
\left\langle
(\widetilde\theta_G^k)_{\beta\beta},
(\widetilde W_G^k)_{\beta\beta}
\right\rangle \\
&\quad
-\gamma_k
\left\langle
(\widetilde W_H^k)_{\beta\beta},
(\widetilde\theta_H^k)_{\beta\beta}
\right\rangle \\
&\quad
+\gamma_k^2
\left\langle
(\widetilde W_H^k)_{\beta\beta},
(\widetilde W_G^k)_{\beta\beta}
\right\rangle
+ \langle R_G^k,R_H^k\rangle .
\end{aligned}
\]
Since $(W_G^k,W_H^k)\in\mathcal W$, we have
\[
W_G^k \succeq 0, \qquad W_H^k \preceq 0, \qquad \langle W_G^k, W_H^k \rangle = 0.
\]
Thus,
\(
\left\langle
(\widetilde W_H^k)_{\beta\beta},
(\widetilde W_G^k)_{\beta\beta}
\right\rangle = 0.
\)
 By Assumption~\ref{ass:comp}, we have
\[
\gamma_k
\left\langle
\theta_G^k, W_G^k
\right\rangle \to 0,
\qquad
\gamma_k
\left\langle
W_H^k, \theta_H^k
\right\rangle \to 0,
\]
and therefore, the corresponding biactive components also vanish asymptotically.
Consequently,
\[
\limsup_{k\to\infty}
\left\langle
(\widetilde\Gamma_G^k)_{\beta\beta},
(\widetilde\Gamma_H^k)_{\beta\beta}
\right\rangle
\le 0,
\]
which proves that $\bar x$ is an AC-stationary point of~\eqref{SDCMPCC}.
\end{proof}

Having established AC-stationarity as a limit property of the generated sequence, we now prove that, under the SDCMPCC-Robinson constraint qualification, this sequential condition implies the exact C-stationarity condition.
\begin{theorem}\label{thm:M-stationarity}
Let $\{x^k\}$ be the sequence generated by
Algorithm~\ref{alg:aug_lagrangian_sdcmpcc_slack}, and suppose that there exists an
infinite subset $K\subset\mathbb N$ such that $x^k\to\bar x$ for $k\in K$, where
$\bar x$ is feasible for~\eqref{SDCMPCC}. If $\bar x$ satisfies SDCMPCC-Robinson,
then $\bar x$ is a C-stationary point of~\eqref{SDCMPCC}.
\end{theorem}

\begin{proof}
By the previous theorem, there exist matrices
$\Gamma_G^{k},\Gamma_H^{k}\in\mathbb S^m$ such that
\begin{equation}\label{eq:stat-app-M}
\nabla f(x^k)
+DG(x^k)^*\Gamma_G^{k}
+DH(x^k)^*\Gamma_H^{k}
\to 0,
\end{equation}
and $(\Gamma_G^{k},\Gamma_H^{k})$ satisfies the AC-stationarity condition.

Let $U\in\mathbb O^m$ be an orthogonal matrix that simultaneously diagonalizes
$G(\bar x)$ and $H(\bar x)$, and let $(\alpha,\beta,\gamma)$ be the corresponding
spectral partition. Define
\[
\widetilde\Gamma_G^{k}:=U^\top\Gamma_G^{k}U,
\qquad
\widetilde\Gamma_H^{k}:=U^\top\Gamma_H^{k}U.
\]

We first show that the sequence
$\{(\Gamma_G^{k},\Gamma_H^{k})\}_{k\in K}$ is bounded. Suppose, by contradiction,
that it is unbounded. Let
\[
t_k:=\|(\Gamma_G^{k},\Gamma_H^{k})\|.
\]
Passing to a subsequence if necessary, we may suppose that
\[
\frac{\Gamma_G^{k}}{t_k}\to\widehat\Gamma_G,
\qquad
\frac{\Gamma_H^{k}}{t_k}\to\widehat\Gamma_H,
\qquad
\|(\widehat\Gamma_G,\widehat\Gamma_H)\|=1.
\]
Dividing~\eqref{eq:stat-app-M} by $t_k$ and letting $k\to\infty$, we obtain
\begin{equation}\label{eq:limit-lin-M}
DG(\bar x)^*\widehat\Gamma_G
+
DH(\bar x)^*\widehat\Gamma_H
=0.
\end{equation}

From the definitions of $\Gamma_G^{k}$ and $\Gamma_H^{k}$ in the slack augmented
Lagrangian scheme, together with the boundedness of the safeguarded multipliers,
the normalized limits can only have possibly nonzero components in the blocks
appearing in the tightened problem. Hence,
\[
(\widehat\Gamma_G)_{\alpha\alpha}=0,
\qquad
(\widehat\Gamma_H)_{\gamma\gamma}=0,
\]
and the only possibly nonzero blocks are
\[
(\widehat\Gamma_G)_{\beta\beta},\ 
(\widehat\Gamma_G)_{\gamma\gamma},\ 
(\widehat\Gamma_G)_{\alpha\beta},\ 
(\widehat\Gamma_G)_{\alpha\gamma},\ 
(\widehat\Gamma_G)_{\beta\gamma},
\]
and
\[
(\widehat\Gamma_H)_{\beta\beta},\ 
(\widehat\Gamma_H)_{\alpha\alpha},\ 
(\widehat\Gamma_H)_{\alpha\beta},\ 
(\widehat\Gamma_H)_{\alpha\gamma},\ 
(\widehat\Gamma_H)_{\beta\gamma}.
\]
Therefore,~\eqref{eq:limit-lin-M} can be rewritten as
\small{\[
A^*
\big(
(\widehat\Gamma_G)_{\beta\beta},
(\widehat\Gamma_G)_{\gamma\gamma},
(\widehat\Gamma_G)_{\alpha\beta},
(\widehat\Gamma_G)_{\alpha\gamma},
(\widehat\Gamma_G)_{\beta\gamma},
(\widehat\Gamma_H)_{\beta\beta},
(\widehat\Gamma_H)_{\alpha\alpha},
(\widehat\Gamma_H)_{\alpha\beta},
(\widehat\Gamma_H)_{\alpha\gamma},
(\widehat\Gamma_H)_{\beta\gamma}
\big)
=0.
\]}\normalsize
Since SDCMPCC-Robinson means that $A$ is surjective, $A^*$ is injective. Hence, all
the blocks above vanish, and consequently
\[
\widehat\Gamma_G=0,
\qquad
\widehat\Gamma_H=0,
\]
contradicting
\[
\|(\widehat\Gamma_G,\widehat\Gamma_H)\|=1.
\]
Thus, $\{(\Gamma_G^{k},\Gamma_H^{k})\}_{k\in K}$ is bounded.
Therefore, passing to a further subsequence if necessary, there exist
$\Gamma_G,\Gamma_H\in\mathbb S^m$ such that
\[
\Gamma_G^{k}\to\Gamma_G,
\qquad
\Gamma_H^{k}\to\Gamma_H.
\]
Passing to the limit in~\eqref{eq:stat-app-M}, we obtain
\[
\nabla f(\bar x)
+DG(\bar x)^*\Gamma_G
+DH(\bar x)^*\Gamma_H
=0.
\]
Moreover, the AW-stationarity conditions yield, by passing to the limit,
\[
(\widetilde\Gamma_G)_{\alpha\alpha}=0,
\qquad
(\widetilde\Gamma_H)_{\gamma\gamma}=0,
\]
where
\[
\widetilde\Gamma_G:=U^\top\Gamma_G U,
\qquad
\widetilde\Gamma_H:=U^\top\Gamma_H U.
\]
Hence, $\bar x$ is W-stationary.

It remains to verify the C-condition on the biactive block. Since
$(\Gamma_G^{k},\Gamma_H^{k})$ satisfies AC-stationarity, we have
\[
\limsup_{k\to\infty}
\left\langle
(\widetilde\Gamma_G^{k})_{\beta\beta},
(\widetilde\Gamma_H^{k})_{\beta\beta}
\right\rangle
\le 0.
\]
Since
\[
\widetilde\Gamma_G^{k}\to\widetilde\Gamma_G,
\qquad
\widetilde\Gamma_H^{k}\to\widetilde\Gamma_H,
\]
and the inner product is continuous, as follows
\[
\left\langle
(\widetilde\Gamma_G)_{\beta\beta},
(\widetilde\Gamma_H)_{\beta\beta}
\right\rangle
=
\lim_{k\to\infty}
\left\langle
(\widetilde\Gamma_G^{k})_{\beta\beta},
(\widetilde\Gamma_H^{k})_{\beta\beta}
\right\rangle
\le 0.
\]
Thus, $\bar x$ satisfies the C-stationarity condition for~\eqref{SDCMPCC}.
\end{proof}

The previous result shows that, under SDCMPCC-Robinson CQ for~\eqref{TSDCMPCC_refined} , AC-stationarity implies exact C-stationarity for~\eqref{SDCMPCC}. In this sense, Robinson’s condition plays the role of a bridge between sequential optimality conditions, which naturally arise from algorithmic schemes, and exact stationarity concepts.

\section{Conclusions}\label{comentarios}

In this work, we investigated nonlinear semidefinite programming problems with complementarity constraints, a class of highly degenerate problems in which classical constraint qualifications typically fail. 

We introduced a tightened reformulation that isolates the complementarity structure through explicit constraints, enabling a structured analysis of first-order optimality conditions. Within this framework, we defined notions of exact and sequential stationarity, namely W- and C-stationarity and their approximate counterparts, AW- and AC-stationarity, extending concepts from MPCC to the semidefinite setting.

A key feature of the proposed framework is that the resulting stationarity conditions are directly connected to the structure of the tightened problem and avoid the cross-block coupling effects that arise in approaches based on the variational analysis of the normal cone. In particular, C-stationarity is characterized by a compatibility condition restricted to the biactive block, leading to a more tractable and structured formulation.

From an algorithmic perspective, we analyzed an augmented Lagrangian strategy and established that every accumulation point satisfies AW-stationarity, as well as AC-stationarity under a mild assumption, and C-stationarity under the SDCMPCC-Robinson constraint qualification. These results provide a theoretical foundation for the use of augmented Lagrangian-type methods in this setting, even in the absence of classical constraint qualifications.

Future research directions include the development of numerical implementations and the extension of the proposed framework to more general symmetric cone complementarity problems. Another promising direction is to explore the ideas introduced in~\cite{guo2022}, where an algorithm closely related to ours is analyzed in the MPCC setting, along with an effective strategy for solving the subproblems arising in the augmented Lagrangian framework.

\section*{Declarations}

This work was funded by FAPESP under contracts no.~2023/01655-2 and 2023/08706-1 and by CNPq under contract no.~407147/2023-3, 403197/2025-2, and 409837/2025-3.

%%===================================================%%
%% For presentation purpose, we have included        %%
%% \bigskip command. Please ignore this.             %%
%%===================================================%%
%\bigskip
%\begin{flushleft}%
%Editorial Policies for:

%\bigskip\noindent
%Springer journals and proceedings: \url{https://www.springer.com/gp/editorial-policies}

%\bigskip\noindent
%Nature Portfolio journals: \url{https://www.nature.com/nature-research/editorial-policies}

%\bigskip\noindent
%\textit{Scientific Reports}: \url{https://www.nature.com/srep/journal-policies/editorial-policies}

%\bigskip\noindent
%BMC journals: \url{https://www.biomedcentral.com/getpublished/editorial-policies}
%\end{flushleft}

\begin{appendices}
\section{Proof of Guignard's Constraint Qualification for the Set $\mathcal{W}$}
This appendix contains the proof of Theorem~\ref{guinard}, which establishes that \emph{Guignard’s constraint qualification} (GCQ) holds for the feasible set $\mathcal{W}$ defined in Section~6. First,
we recall some notions from variational analysis that will be used in the proof of our main results; see~\cite{RockafellarWets1998} for further details. 
Throughout this appendix, $\mathbb{E}$ denotes a finite-dimensional Euclidean space endowed with the inner product $\langle\cdot,\cdot\rangle$. For a set $C\subset\mathbb{E}$, its polar cone is defined by
$
C^\circ := \{ v\in\mathbb{E} \mid \langle v,c\rangle \le 0 \;\; \forall c\in C \}.
$

\begin{definition}
Let $C\subset\mathbb{E}$ be a closed set and let $\bar z\in C$. 
The tangent cone to $C$ at $\bar z$ is defined as
$
T_C(\bar z) := \left\{ d\in\mathbb{E} \;\middle|\; \exists\, t_k \downarrow 0,\; \exists\, d_k \to d \text{ such that } \bar z + t_k d_k \in C \;\; \forall k \right\}.
$
\end{definition}

The tangent cone represents the collection of directions that are compatible with the geometry of the set $C$ at the point $\bar z$ and can be interpreted as the set of admissible first-order variations. In particular, when $C=\mathbb{S}^m_+$ and $\bar z = M \in \mathbb{S}^m_+$, the tangent cone admits a characterization in terms of the kernel of $M$. More precisely, a symmetric matrix $N \in \mathbb{S}^m$ belongs to $T_{\mathbb{S}^m_+}(M)$ if and only if it preserves nonnegativity along the null space of $M$, that is,
$
T_{\mathbb{S}^m_+}(M) = \left\{ N \in \mathbb{S}^m \;\middle|\; d^\top N d \ge 0 \;\; \text{for all } d \in \ker M \right\}.
$
\begin{definition}
A closed set $C\subset\mathbb{E}$ is said to be \emph{Clarke regular} at $\bar z\in C$ if its limiting and regular normal cones coincide at $\bar z$; see, e.g.,~\cite{RockafellarWets1998} for the explicit definitions of these normal cones.
\end{definition}

Clarke regularity ensures that the geometric structure of the set is sufficiently well behaved so that tangent and normal constructions admit simple characterizations. 
In particular, all closed convex sets and all sets defined by $C^1$ equality or inequality constraints are Clarke regular; see~\cite[Example~6.14 and Proposition~6.13]{RockafellarWets1998}.
A fundamental consequence of Clarke regularity is the validity of the intersection rule for tangent cones.

\begin{proposition}[{\cite[Theorem~6.42]{RockafellarWets1998}}]
Let $C_1,\dots,C_m\subset\mathbb{E}$ be closed sets that are Clarke regular at $\bar z\in\bigcap_{i=1}^m C_i$. Then,
\[
T_{\cap_{i=1}^m C_i}(\bar z) = \bigcap_{i=1}^m T_{C_i}(\bar z).
\]
\end{proposition}
We next recall a regularity property for sets defined by smooth inequality constraints, which will be crucial in the analysis of complementarity-type conditions.

\begin{lemma}[{\cite[Proposition~6.5 and Theorem~6.31]{RockafellarWets1998}}]\label{clarke}
Let $\varphi:\mathbb{E}\to\mathbb{R}$ be continuously differentiable and define
\(
C := \{ z\in\mathbb{E} \mid \varphi(z) \ge 0 \}.
\)
If $\nabla \varphi(\bar z) \neq 0$, then the set $C$ is Clarke regular at $\bar z$, and its tangent cone is given by
\[
T_C(\bar z)=
\begin{cases}
\mathbb{E}, & \text{if } \varphi(\bar z)>0,\\[2mm]
\{ d\in\mathbb{E} \mid \langle \nabla \varphi(\bar z), d \rangle \ge 0 \},
& \text{if } \varphi(\bar z)=0.
\end{cases}
\]
\end{lemma}

Finally, we recall a basic property of normal cones to Cartesian product sets, which will be used when dealing with matrix variables and product spaces.

\begin{lemma}
Let $C_1,C_2\subset\mathbb{E}$ be closed convex sets. Then, for any $(x_1,x_2)\in C_1\times C_2$, it holds that
$
N_{C_1\times C_2}(x_1,x_2)
=
N_{C_1}(x_1)\times N_{C_2}(x_2).
$
\end{lemma}

This decomposition property allows us to treat normal cones of product sets componentwise and will be repeatedly used in the sequel. To introduce Guignard's constraint qualification in the context of the problem~\eqref{NLSDP}, we first recall some basic constructions from cone geometry.
We denote the feasible set of~\eqref{NLSDP} by
$$
\mathcal{F} := \{ x \in \mathbb{R}^n \mid H(x) \preceq 0,\ P(x) = 0 \},
$$
and consider a feasible point $\bar{x} \in \mathcal{F}$. The (Bouligand) tangent cone to $\mathcal{F}$ at $\bar x$ is denoted by $T_{\mathcal{F}}(\bar x)$. 
Moreover, the linearized tangent cone to $\mathcal{F}$ at $\bar x$ is given by
\[
L_{\mathcal{F}}(\bar x)
:=
\Big\{ d \in \mathbb{R}^n \ \Big| \ DH(\bar x)d \in T_{\mathbb{S}^m_-}(H(\bar x)), \; DP(\bar x)d = 0 \Big\},
\]
where $T_{\mathbb{S}^m_-}(H(\bar x))$ denotes the tangent cone to the closed convex cone $\mathbb{S}^m_-$ at $H(\bar x)$.
 The relevance of these cones stems from the fact that, if $\bar{x}$ is a local minimizer of~\eqref{NLSDP}, the following \emph{first-order geometric necessary condition} holds:
\begin{equation}\label{eq:FO_geom}
-\nabla f(\bar{x}) \in T_{\mathcal{F}}(\bar{x})^\circ.
\end{equation}  
Although condition~\eqref{eq:FO_geom} is conceptually simple, it is rarely tractable in practice, since the tangent cone $T_{\mathcal{F}}(\bar{x})$ may not admit a tractable explicit characterization.  
In contrast, the polar of the linearized tangent cone can be expressed explicitly, as shown in the following adaptation of a result by Guignard~\cite{andreani2023first}.
\begin{lemma}\label{lem:guignard}
Let $\bar{x} \in \mathcal{F}$. Then,
\[
L_{\mathcal{F}}(\bar{x})^\circ = \operatorname{cl}\big( \mathcal{H}(\bar{x}) \big),
\]
where
$
\mathcal{H}(\bar{x}) := DH(\bar{x})^* \mathcal{N}_{\mathbb{S}^m_-}(H(\bar{x})) 
\ + \ DP(\bar{x})^* \mathbb{R}^{p\times q},
$
and $\mathcal{N}_{\mathbb{S}^m_-}(H(\bar{x})) := T_{\mathbb{S}^m_-}(H(\bar{x}))^\circ$ denotes the normal cone to $\mathbb{S}^m_-$ at $H(\bar{x})$.
\end{lemma}

Here, the term $DP(\bar{x})^* \mathbb{R}^{p\times q}$ corresponds to the contribution of the matrix equality constraint $P(x)=0$ to the polar cone.  
Combining the geometric condition~\eqref{eq:FO_geom} with Lemma~\ref{lem:guignard} yields the following theorem.

\begin{theorem}\label{thm:guignard}
Let $\bar{x} \in \mathcal{F}$ be a local minimizer of~\eqref{NLSDP}.  
If
\[
T_{\mathcal{F}}(\bar{x})^\circ = L_{\mathcal{F}}(\bar{x})^\circ
\quad \text{and} \quad \mathcal{H}(\bar{x}) \ \text{is closed},
\]
then there exist $(\bar{\Omega}, \bar{\Lambda}) \in \mathbb{S}^m_+ \times \mathbb{R}^{p\times q}$ such that
\[
\nabla f(\bar{x}) + DH(\bar{x})^* \bar{\Omega} + DP(\bar{x})^* \bar{\Lambda} = 0,
\quad \langle H(\bar{x}), \bar{\Omega} \rangle = 0.
\]
\end{theorem}
Condition
$T_{\mathcal{F}}(\bar{x})^\circ = L_{\mathcal{F}}(\bar{x})^\circ
\quad \text{and} \quad \mathcal{H}(\bar{x}) \ \text{closed}
$
is known as \emph{Guignard's constraint qualification} (GCQ).  
It is the weakest condition ensuring that the KKT conditions are necessary for local optimality: if the KKT system holds at $\bar{x}$ for every $C^1$ objective $f$ having $\bar{x}$ as a local minimizer over $\mathcal{F}$, then GCQ must also hold at $\bar{x}$, see~\cite{andreani2023first}.
 
 We now show that Guignard's constraint qualification holds for the complementarity set
$$
\mathcal{W} = \left\{ (W_G, W_H) \in \mathbb{S}^m_{+} \times \mathbb{S}^m_{-} \;\middle|\;
 \langle W_G, W_H \rangle \geq 0 \right\}.
$$
\begin{theorem}\label{guinard}
Let $(W_G,W_H)\in \mathcal{W}$. Then Guignard's constraint qualification holds at $(W_G,W_H)$ $$(T_{\mathcal{W} }(W_G,W_H)^\circ = L_{\mathcal{W} }(W_G,W_H)^\circ)$$
and the set $\mathcal H(W_G,W_H)$ is closed.
\end{theorem}

\begin{proof}
We work in the Euclidean space $\mathbb E := \mathbb S^m \times \mathbb S^m$ endowed with the inner product
\[
\langle (A,B),(X,Y)\rangle := \langle A,X\rangle + \langle B,Y\rangle .
\]

\noindent (i) Linearized cone. 

Define  $\Phi:\mathbb E \to \mathbb S^m \times \mathbb S^m \times \mathbb R$ by
\(
\Phi(W_G,W_H) := (W_G,W_H,\langle W_G,W_H\rangle),
\)
and set
\(
\mathcal{K} := \mathbb S^m_+ \times \mathbb S^m_- \times \mathbb R_+ .
\)
Then
\(
\mathcal{W} = \{(W_G,W_H)\in \mathbb E \mid \Phi(W_G,W_H)\in \mathcal{K}\}.
\)
The derivative of $\Phi$ at $(W_G,W_H)$ in the direction $(\Delta_G,\Delta_H)$ is given by
$
D\Phi(W_G,W_H)(\Delta_G,\Delta_H)
= (\Delta_G,\Delta_H,\langle \Delta_G,W_H\rangle + \langle W_G,\Delta_H\rangle).
$
Hence,
\begin{equation*}
\begin{aligned}
    L_{\mathcal W}(W_G,W_H)
=
\{&(\Delta_G,\Delta_H)\in\mathbb E :
\Delta_G \in T_{\mathbb S^m_+}(W_G),\\
&\Delta_H \in T_{\mathbb S^m_-}(W_H),\ 
\langle \Delta_G,W_H\rangle + \langle W_G,\Delta_H\rangle \ge 0\}.
\end{aligned}
\end{equation*}

\noindent (ii)  Guignard's constraint qualification.  In order to establish the Guignard's CQ it is sufficient to prove that
$T_{\mathcal{W}}(W_G,W_H)^\circ\subset 
L_{\mathcal{W}}(W_G,W_H)^\circ.
$

We distinguish two cases.

\noindent Case 1: $(W_G,W_H)\neq (0,0)$.
Define $\varphi(W_G,W_H):=\langle W_G,W_H\rangle$. Then
\(
\nabla \varphi(W_G,W_H) = (W_H,W_G) \neq 0.
\)
Since $\mathbb S_+^m$, $\mathbb S_-^m$ and the level set
\(
C_\varphi := \{(W_G,W_H)\mid \varphi(W_G,W_H)\ge 0\}
\)
are closed and Clarke regular, the tangent cone to $\mathcal W$ at $(W_G,W_H)$ can be written as the intersection of the corresponding tangent cones, namely,
\(
T_{\mathcal W}(W_G,W_H)
=
\big(T_{\mathbb S_+^m}(W_G)\times T_{\mathbb S_-^m}(W_H)\big)\cap T_{C_\varphi}(W_G,W_H).
\)
Moreover, by  Lemma A.4, the tangent cone to $C_\varphi$ is given by
\[
T_{C_\varphi}(W_G,W_H)
=
\begin{cases}
\mathbb S^m \times \mathbb S^m, & \text{if } \langle W_G,W_H\rangle > 0,\\[0.2cm]
\{(\Delta_G,\Delta_H)\mid \langle \Delta_G,W_H\rangle + \langle W_G,\Delta_H\rangle \ge 0\}, 
& \text{if } \langle W_G,W_H\rangle = 0.
\end{cases}
\]
Therefore,
\small{\[
T_{\mathcal W}(W_G,W_H)
=
\Big\{(\Delta_G,\Delta_H)\,\Big|
\Delta_G \in T_{\mathbb S_+^m}(W_G),\,
\Delta_H \in T_{\mathbb S_-^m}(W_H),\,
\langle \Delta_G,W_H\rangle + \langle W_G,\Delta_H\rangle \ge 0
\Big\}
\]}\normalsize
whenever $\langle W_G,W_H\rangle=0$. Comparing this characterization with the expression of the linearized cone obtained in part (i), we conclude that
\(
L_{\mathcal{W}}(W_G,W_H)=T_{\mathcal{W}}(W_G,W_H).
\)

\noindent Case 2: $(W_G,W_H)=(0,0)$.
In this case,  the linearized cone reduces to
\(
L_{\mathcal{W}}(0,0)=\mathbb S^m_+ \times \mathbb S^m_- .
\) We now compute the tangent cone $T_{\mathcal W}(0,0)$.
We recall the definition of the tangent cone to a closed set $\mathcal W \subset \mathbb S^m \times \mathbb S^m$ at a point $(W_G,W_H)\in\mathcal W$:
\[
T_{\mathcal W}(W_G,W_H)
:=
\left\{
(\Delta_G,\Delta_H)\;\middle|\;
\exists\, t_k \downarrow 0 \text{ such that } (W_G,W_H)+t_k(\Delta_G,\Delta_H)\in \mathcal W \ \forall k
\right\}.
\]
Let $(\Delta_G,\Delta_H)\in T_{\mathcal W}(0,0)$. Then there exists $t_k \downarrow 0$ such that $(t_k\Delta_G,t_k\Delta_H)\in \mathcal W$ for all $k$. Hence, $t_k\Delta_G\succeq0$, $t_k\Delta_H\preceq0$, and $\langle t_k\Delta_G,t_k\Delta_H\rangle\ge0$, which implies
\[
\Delta_G\succeq0,\qquad \Delta_H\preceq0,\qquad \langle \Delta_G,\Delta_H\rangle\ge0.
\]

On the other hand, $\Delta_G \succeq 0$ and $\Delta_H \preceq 0$ yield
$\langle \Delta_G,\Delta_H\rangle \le 0$.
Therefore,
\(
\langle \Delta_G,\Delta_H\rangle = 0.
\)
Hence, the tangent cone admits the characterization
\[
T_{\mathcal W}(0,0)
=
\{(\Delta_G,\Delta_H)\mid 
\Delta_G \succeq 0,\ \Delta_H \preceq 0,\ \langle \Delta_G,\Delta_H\rangle = 0\}.
\]
We now compute the polar cones of $L_{\mathcal W}(0,0)$ and $T_{\mathcal W}(0,0)$.
Recall that, for a cone $C\subset \mathbb S^m\times \mathbb S^m$, its polar cone is
\[
C^\circ := \{(A,B)\in \mathbb S^m\times \mathbb S^m \mid 
\langle A,\Delta_G\rangle + \langle B,\Delta_H\rangle \le 0,\ \forall (\Delta_G,\Delta_H)\in C\}.
\]
Since $L_{\mathcal W}(0,0)=\mathbb S_+^m\times \mathbb S_-^m$ is a Cartesian product of closed convex cones,
\(
L_{\mathcal W}(0,0)^\circ 
= (\mathbb S_+^m)^\circ \times (\mathbb S_-^m)^\circ.
\)
Using $(\mathbb S_+^m)^\circ=\mathbb S_-^m$ and $(\mathbb S_-^m)^\circ=\mathbb S_+^m$, we obtain
\(
L_{\mathcal W}(0,0)^\circ = \mathbb S_-^m \times \mathbb S_+^m.
\)

Let $(A,B)\in T_{\mathcal W}(0,0)^\circ$.  
Choosing $(\Delta_G,\Delta_H)=(\Delta_G,0)$ with arbitrary $\Delta_G\succeq0$ (which belongs to $T_{\mathcal W}(0,0)$),
we obtain
\[
\langle A,\Delta_G\rangle \le 0, \quad \forall\, \Delta_G\succeq0,
\]
hence, $A\in \mathbb S_-^m$.
Similarly, choosing $(\Delta_G,\Delta_H)=(0,\Delta_H)$ with arbitrary $\Delta_H\preceq0$ yields
\[
\langle B,\Delta_H\rangle \le 0, \quad \forall\, \Delta_H\preceq0,
\]
hence, $B\in \mathbb S_+^m$.
Therefore,
\(
T_{\mathcal W}(0,0)^\circ \subset \mathbb S_-^m\times \mathbb S_+^m.
\)

Conversely, let $(A,B)\in \mathbb S_-^m\times \mathbb S_+^m$ and $(\Delta_G,\Delta_H)\in T_{\mathcal W}(0,0)$. Then $\Delta_G\succeq0$, $\Delta_H\preceq0$, $A\preceq0$, and $B\succeq0$, so that $
\langle A,\Delta_G\rangle \le 0$, $\langle B,\Delta_H\rangle \le 0$,
and hence, $\langle A,\Delta_G\rangle+\langle B,\Delta_H\rangle \le 0$. This shows $(A,B)\in T_{\mathcal W}(0,0)^\circ$, and therefore
$
\mathbb S_-^m\times \mathbb S_+^m \subset T_{\mathcal W}(0,0)^\circ.
$
Combining both inclusions,
\(
T_{\mathcal W}(0,0)^\circ = \mathbb S_-^m\times \mathbb S_+^m = L_{\mathcal W}(0,0)^\circ.
\)

\noindent(iii) Closedness of $\mathcal H(W_G,W_H)$.
Recall that
\[
\mathcal H(W_G,W_H)=D\Phi(W_G,W_H)^*\,\mathcal N_{\mathcal K}(\Phi(W_G,W_H)),
\]
where $\Phi(W_G,W_H)=(W_G,W_H,\langle W_G,W_H\rangle)$ and
$\mathcal K=\mathbb S^m_+\times\mathbb S^m_-\times\mathbb R_+$.
Since $\mathcal K$ is a product cone, we have for $(U,V,\sigma)\in\mathcal K$,
\[
\mathcal N_{\mathcal K}(U,V,\sigma)
=
\mathcal N_{\mathbb S^m_+}(U)\times \mathcal N_{\mathbb S^m_-}(V)\times \mathcal N_{\mathbb R_+}(\sigma).
\]
Moreover, a direct computation shows that for $(\Xi_G,\Xi_H,\tau)\in\mathbb S^m\times\mathbb S^m\times\mathbb R$,
\[
D\Phi(W_G,W_H)^*(\Xi_G,\Xi_H,\tau)=(\Xi_G+\tau W_H,\ \Xi_H+\tau W_G).
\]
Hence, letting $\sigma:=\langle W_G,W_H\rangle$, we can write
\begin{equation}\label{eq:H-def}
\begin{aligned}
\mathcal H(W_G,W_H)
=
&\Big\{(\Xi_G+\tau W_H,\ \Xi_H+\tau W_G)\ \Big|\, 
\Xi_G\in \mathcal N_{\mathbb S^m_+}(W_G),\, \\
&\Xi_H\in \mathcal N_{\mathbb S^m_-}(W_H),\, 
\tau\in \mathcal N_{\mathbb R_+}(\sigma)\Big\}.
\end{aligned}
\end{equation}
We now prove that $\mathcal H(W_G,W_H)$ is closed. 

Suppose that $\sigma=0$. Then $\mathcal N_{\mathbb R_+}(0)=\mathbb R_-$ and~\eqref{eq:H-def} give
\begin{equation}\label{eq:H-sum}
\mathcal H(W_G,W_H)
=
\big(\mathcal N_{\mathbb S^m_+}(W_G)\times \mathcal N_{\mathbb S^m_-}(W_H)\big)
+\mathbb R_-(W_H,W_G).
\end{equation}
Since $(W_G,W_H)\in\mathcal W$, we have $W_G\succeq0$, $W_H\preceq0$, and $\langle W_G,W_H\rangle=0$. Hence, $W_G$ and $W_H$ are simultaneously diagonalizable: there exists $U\in\mathcal O_m$ such that
\[
W_G = U\,\mathrm{diag}(\lambda(W_G))\,U^\top,\qquad
W_H = U\,\mathrm{diag}(\lambda(W_H))\,U^\top,
\]
with $\lambda_i(W_G)\ge 0$, $\lambda_i(W_H)\le 0$, and $\lambda_i(W_G)\lambda_i(W_H)=0$ for all $i=1,\dots,m$.

Let $\mathcal Q:\mathbb E\to\mathbb E$ be the linear isometry
$\mathcal Q(A,B)=(U^\top A U,\ U^\top B U)$. Since $\mathcal Q$ is a linear isometry,
$\mathcal H(W_G,W_H)$ is closed if and only if $\mathcal Q(\mathcal H(W_G,W_H))$ is closed.
Applying $\mathcal Q$ to~\eqref{eq:H-sum} and using the invariance of the normal cone of the semidefinite cones under orthogonal congruence, we obtain
\begin{equation*}
\begin{aligned}
\mathcal Q(\mathcal H(W_G,W_H)) =
\ &\Big(\mathcal N_{\mathbb S^m_+}(\operatorname{diag}(\lambda(W_G)))\times 
\mathcal N_{\mathbb S^m_-}(\operatorname{diag}(\lambda(W_H)))\Big)\\ 
&+\ \mathbb R_-(\operatorname{diag}(\lambda(W_H)),\operatorname{diag}(\lambda(W_G))).
\end{aligned}
\end{equation*}

In these spectral coordinates, the above set decomposes as a finite Cartesian product over $i=1,\dots,m$ of
two-dimensional sets, namely
\[
\mathcal Q(\mathcal H(W_G,W_H))
=
\prod_{i=1}^m \mathcal H_i(\lambda_i(W_G),\lambda_i(W_H)),
\]
where
\small{\[
\mathcal H_i(\lambda_i(W_G),\lambda_i(W_H))
:=
\left\{
(\lambda_G+\tau \lambda_i(W_H),\ \lambda_H+\tau \lambda_i(W_G))
\ \middle|\ 
\begin{aligned}
&\lambda_G\in \mathcal N_{\mathbb R_+}(\lambda_i(W_G)),\\
&\lambda_H\in \mathcal N_{\mathbb R_-}(\lambda_i(W_H)),\\
&\tau\le 0
\end{aligned}
\right\}.
\]}
\normalsize

Clearly, $\mathcal H_i(\lambda_i(W_G),\,\lambda_i(W_H))\subset\mathbb R^2$.
Since we have $\lambda_i(W_G)\ge 0$, $\lambda_i(W_H)\le 0$, and $\lambda_i(W_G)\lambda_i(W_H)=0$, only the following cases may occur:

\noindent\emph{Case A: $\lambda_i(W_G)>0$, $\lambda_i(W_H)=0$.}
Then $\mathcal N_{\mathbb R_+}(\lambda_i(W_G))=\{0\}$ and $\mathcal N_{\mathbb R_-}(0)=\mathbb R_+$, hence
\[
\mathcal H_i(\lambda_i(W_G),0)=\{(0,\lambda_H+\tau \lambda_i(W_G))\mid \lambda_H\ge0,\ \tau\le0\}
=\{(0,t)\mid t\in\mathbb R\},
\]
which is closed.

\noindent\emph{Case B: $\lambda_i(W_G)=0$, $\lambda_i(W_H)<0$.}
Then $\mathcal N_{\mathbb R_+}(0)=\mathbb R_-$ and $\mathcal N_{\mathbb R_-}(\lambda_i(W_H))=\{0\}$, hence
\[
\mathcal H_i(0,\lambda_i(W_H))=\{(\lambda_G+\tau \lambda_i(W_H),0)\mid \lambda_G\le0,\ \tau\le0\}
=\{(t,0)\mid t\in\mathbb R\},
\]
which is closed.

\noindent\emph{Case C: $\lambda_i(W_G)=0$, $\lambda_i(W_H)=0$.}
Then $\mathcal N_{\mathbb R_+}(0)=\mathbb R_-$ and $\mathcal N_{\mathbb R_-}(0)=\mathbb R_+$, hence
\[
\mathcal H_i(0,0)=\{(\lambda_G,\lambda_H)\mid \lambda_G\le0,\ \lambda_H\ge0\}
=\mathbb R_-\times\mathbb R_+,
\]
which is closed.

Therefore, $\mathcal H_i(\lambda_i(W_G),\lambda_i(W_H))$ is closed for every $i$. Since a finite Cartesian product of closed sets is closed, as follows,
$\mathcal Q(\mathcal H(W_G,W_H))$ is closed, and consequently so is $\mathcal H(W_G,W_H)$.
Combining $T_{\mathcal{W}}=L_{\mathcal{W}}$ with the closedness of $\mathcal{H}(W_G,W_H)$ yields Guignard’s Constraint Qualification at $(W_G,W_H)$.
\end{proof}

%%=============================================%%
%% For submissions to Nature Portfolio Journals %%
%% please use the heading ``Extended Data''.   %%
%%=============================================%%

%%=============================================================%%
%% Sample for another appendix section			       %%
%%=============================================================%%

%% \section{Example of another appendix section}\label{secA2}%
%% Appendices may be used for helpful, supporting or essential material that would otherwise 
%% clutter, break up or be distracting to the text. Appendices can consist of sections, figures, 
%% tables and equations etc.

\end{appendices}

%%===========================================================================================%%
%% If you are submitting to one of the Nature Portfolio journals, using the eJP submission   %%
%% system, please include the references within the manuscript file itself. You may do this  %%
%% by copying the reference list from your .bbl file, paste it into the main manuscript .tex %%
%% file, and delete the associated \verb+\bibliography+ commands.                            %%
%%==================================================================================

%\bibliography{sn-bibliography}% common bib file
%% if required, the content of .bbl file can be included here once bbl is generated
%%\input sn-article.bbl

\bibliographystyle{plainnat}
\bibliography{references}
\end{document}